\documentclass[11pt]{amsart}
\usepackage{amsmath}
\usepackage{amsthm}
\usepackage{amssymb}
\usepackage{mathrsfs}
\usepackage{ifthen}
\usepackage{graphicx}
\usepackage[T1]{fontenc} %skandit
\usepackage{enumerate}
\usepackage{hyperref}
\usepackage{float}
\usepackage{tikz}
\usepackage{enumitem}
\usepackage{fullpage}
\usepackage{enumerate}
\usepackage{enumitem}
\usepackage{xcolor}
\usepackage{todonotes}
\usepackage{multicol}
\usepackage{setspace}
 \nonstopmode \numberwithin{equation}{section}

\theoremstyle{plain}

\newtheorem{Thm}{Theorem}[section]
\newtheorem{cor}[Thm]{Corollary}

\newtheorem{rem}[Thm]{Remark}

\theoremstyle{definition}

\newtheorem{defn}[Thm]{Definition}

\newtheorem{matrixcondition}[Thm]{Condition}

\newcommand{\Tr}{\operatorname{Tr}}

\allowdisplaybreaks

\title{Spectrum of Random Matrices with
Exploding Moments}
\author{Indrajit Jana}
\address{Indian Institute of Technology, Bhubaneswar}
\email{ijana@iitbbs.ac.in}
\thanks{Indrajit Jana - \textit{Email: ijana@iitbbs.ac.in;}
}

\author{Sunita Rani}
\address{Indian Institute of Technology, Bhubaneswar}
\email{s21ma09007@iitbbs.ac.in}
\thanks{Sunita Rani - \textit{Email: s21ma09007@iitbbs.ac.in;} (Corresponding author).
}

\begin{document}
\date{}

\begin{abstract} We study the central limit theorem (CLT) for linear eigenvalue statistics of several types of matrix models, whose entries are having exploding moments, i.e., moments of the entries are increasing with the size of the matrix. In particular, we study elliptic, centrosymmetric, circulant, and inter-correlated block matrices. The CLTs are established using asymptotic Wick formula. 


MSC2020 Subject Classification: 15B52; 60BXX; 60FXX
\end{abstract}

\maketitle

%%%%%%%%%%%%%%%%%%%%%%%%%%%%%%%%%%%%%%%%%%%%%%%%%%%%%%%%%%%%%%%%%%%%%%%%%%%
%%%%%%%%%%%%%%%%%%%%%%%%%%%%  Introduction  %%%%%%%%%%%%%%%%%%%%%%%%%%%%%%%
%%%%%%%%%%%%%%%%%%%%%%%%%%%%%%%%%%%%%%%%%%%%%%%%%%%%%%%%%%%%%%%%%%%%%%%%%%%

\section{Introduction} 
%%%%%%%%%%%%%%%%%%%%%%%%%%%%%%%%%%%%%%%%%%%%%%%%%%%%%%%%%%%%%%%%%%%%%%%%%%%
%%%%%%%%%%%%%%%%%%%%%%%%%%%%  Introduction  %%%%%%%%%%%%%%%%%%%%%%%%%%%%%%%
%%%%%%%%%%%%%%%%%%%%%%%%%%%%%%%%%%%%%%%%%%%%%%%%%%%%%%%%%%%%%%%%%%%%%%%%%%%

We consider different types of random matrix models with exploding moments. We study the asymptotic behavior of normalized traces and the fluctuations of the associated linear eigenvalue statistic for monomial test functions. In particular, we prove a central limit theorem (CLT) for these generalized moments under suitable growth conditions. It is observed that although Gaussian fluctuations may still occur, the normalization and the variance depend on the rate at which the moments increase. The CLT is proved by establishing an asymptotic Wick formula for the product of linear eigenvalue statistics. 

We carry out this analysis for several structured random matrix models, including elliptic, centrosymmetric, circulant, inter-correlated block matrices,  %elliptic, inter-correlated, centrosymmetric, and circulant matrices, 
whose CLTs were established in \cite{o2016central, jana2024spectrum, adhikari2018universality, jana2025cltles} when the entries of the matrices were light-tailed. We believe that the techniques used for circulant matrices with exploding moments can be utilized to prove CLT results for some other types of structured random matrices with exploding moments, such as reverse circulant, symmetric circulant, Toeplitz, and Hankel matrices.

Fluctuation results for heavy-tailed random matrices have been studied in various frameworks; see, for example, \cite{BenaychGeorgesGuionnet2014, male2017limiting, benaych2016fluctuations}. Our proof relies on the techniques demonstrated by Male et al. \cite{male2017limiting, benaych2014central}, where they established CLT for exploding moment and heavy tailed Wigner type random matrices. However, when the matrix model is changed from Wigner to other types such as Elliptic, Centrosymmetric, the internal combinatorics to establish Wick's formula ---  the central part of the CLT proof --- changes significantly. Moreover, it also impacts the final covariance kernel. Besides, we use a slightly generalized definition of exploding moment, where the rate of explosion is controlled by an additional parameter $\alpha>0.$ Additionally, we also consider a strongly structured matrix model, namely the circulant random matrix. In this latter case, the techniques used for proving the CLT are different from the techniques used in the former cases. The stringent structure of the circulant matrices, which was analyzed in the context of CLT by Adhikari et al. \cite{adhikari2018universality} in the light tail case, makes it stand out from the Wigner or Elliptic types of matrices. But the techniques used in \cite{adhikari2017fluctuations, benaych2014central} are not directly applicable here.

This article is split into several sections, where each section is dedicated to a particular type of matrix model. Also, for the convenience of the reader, the sections are mostly made self-contained.

% \begin{figure}
% \centering
% \begin{minipage}{0.48\textwidth}
%     \centering
%     \includegraphics[width=\linewidth]{Ellpitic Law (Exploding moments).png}
%     \vspace{2mm}
    
%     (a)
% \end{minipage}
% \hfill
% \begin{minipage}{0.48\textwidth}
%     \centering
%     \includegraphics[width=\linewidth]{circular law for non hermtian matrices with exploding moments.png}
%     \vspace{2mm}
    
%     (b)
% \end{minipage}

% \vspace{4mm}
% \caption{
% (a) Scatter plot of the real and imaginary parts of eigenvalues of a $6000\times 6000$ elliptic random matrix with exploding moments.  \\
% (b) Scatter plot of the real and imaginary parts of eigenvalues of a $6000\times 6000$ non-Hermitian random matrix with exploding moments.
% }
% \label{fig:combined}
% \end{figure}
%%%%%%%%%%%%%%%%%%%%%%%%%%%%%%%%%%%%%%%%%%%%%%%%%%%%%%%%%%%%%%%%%%%%%%%%%%%
%%%%%%%%%%%%%%%%%%     Elliptic matrices with exploding moments   %%%%%%%%%%%%%%%%%%%%%
%%%%%%%%%%%%%%%%%%%%%%%%%%%%%%%%%%%%%%%%%%%%%%%%%%%%%%%%%%%%%%%%%%%%%%%%%%%

%\section{Matrix Model and Assumptions}
\section{Elliptic matrices}

We consider a sequence of non-Hermitian random matrices
$$
A = \frac{1}{\sqrt{N}} X = [a_{ij}]_{1 \le i,j \le N},
$$
whose entries satisfy the correlation and exploding moment conditions stated in the Condition \ref{cond:matrixcond_1}. Here $X = [x_{ij}]_{1 \le i,j \le N}$ is a real-valued non-Hermitian random matrix.

\begin{matrixcondition}[Elliptic random matrices with exploding moments]
\label{cond:matrixcond_1}
Let $A = \frac{1}{\sqrt{N}} X = [a_{ij}]_{1 \le i,j \le N},$ where $X = [x_{ij}]_{1 \le i,j \le N}$ satisfies the following conditions.

\begin{enumerate}
\item   For all $1 \le i,j \le N$,
       $$
       \mathbb{E}[x_{ij}] = 0, \qquad 
       \mathbb{E}[x_{ij}^2] = 1.
      $$
\item The diagonal entries of $X$ are independent among themselves and of off-diagonal entries as well. The off-diagonal entries' only dependence occurs between pairs $(x_{ij},x_{ji}),$ while $(x_{ij},x_{ji})$ is independent of $(x_{kl},x_{lk})$ if $\{i, j\}\neq \{k, l\}.$
\item The entries of $X$ have exploding moments in the sense that for all $k,l \ge 0$ with $k+l \ge 2$,
$$
\frac{\mathbb{E}[x_{ij}^{k}x_{ji}^{l}]}{N^{(k+l)/2-\alpha}} \longrightarrow C_{k,l},
$$
for $i<j$. For all $k \ge 1$,
$$
%\frac{\mathbb{E}[x_{ii}^{k}]}{N^{k/2-\alpha}} = O(1), \qquad 
\frac{\mathbb{E}[|x_{ii}|^{k}]}{N^{k/2-\alpha}} = O(1),
$$
for some $\alpha > 0$, where $(C_{k,l})_{k,l \ge 0}$ are finite constants.
\end{enumerate}
\end{matrixcondition}

\begin{rem}
   The matrix ensemble considered here belongs to the class of random matrices with exploding moments, which is quantified by the factor $N^{k/2-\alpha}$. In particular, the factor
$N^{- \alpha}$ shows the moment growth does not depend on $k$, and therefore this is not just a rescaled version of a light-tailed model.

The parameter $\alpha$ controls how fast the moments grow with $N$. When $\alpha = 0$, the model reduces to the classical light-tailed setting with bounded moments. For $\alpha > 0$, the higher moments increase with $N$, which is characteristic of heavy-tailed behavior. Moreover, larger values of $\alpha$ mean the heavy-tailed effect becomes stronger.
\end{rem}

%%%%%%%%%%%%%%%%%%%%%%%%%%%%%%%%%%%%%%%%%%%%%%%%%%%%%%%%%%%%%%%%%%%%%%%%%%%
%%%%%%%%%%%%%%%%%%%%%%%%%%%%  Main result  %%%%%%%%%%%%%%%%%%%%%%%%%%%%%%%
%%%%%%%%%%%%%%%%%%%%%%%%%%%%%%%%%%%%%%%%%%%%%%%%%%%%%%%%%%%%%%%%%%%%%%%%%%%

%\subsection{Main result}\label{sec:main result}

\subsection{Convergence of the normalized traces}\label{section:Central Limit Theorem for Generalized Moments for corr matrices}
Let $k \geq 1$ be an integer. The normalized trace of $A^k$ can be written as
\begin{align}\label{eq:trace_expand}
\frac{1}{N}\Tr(A^k)
&= \frac{1}{N} \sum_{i_1,\dots,i_k = 1}^N 
a_{i_1 i_2} a_{i_2 i_3} \cdots a_{i_k i_1}.
\end{align}
To study the asymptotic behavior of this trace, we consider $\mathcal{P}(k)$, which denotes the set of partitions of $\{1,2,\dots,k\}$.
For each $\pi \in \mathcal{P}(k)$, define $S_\pi \subset [N]^k$ as the set of all multi indices $(i_1,i_2\dots,i_k)$ satisfying $i_{m}=i_{n}$ iff $m\sim_{\pi}n.$
Thus, the expansion \eqref{eq:trace_expand} can be rewritten as
\begin{align}\label{eq:trace_partition}
    \frac{1}{N}\Tr(A^k)
%     &=\frac{1}{N} \sum_{i_1,\dots,i_k = 1}^N 
% a_{i_1 i_2} a_{i_2 i_3} \cdots a_{i_k i_1}\notag\\
&=\sum_{\pi \in \mathcal{P}(k)}\frac{1}{N}\sum_{\mathbf{i}\in S_{\pi}}a_{i_1 i_2} a_{i_2 i_3} \cdots a_{i_k i_1}\notag\\
&=: \sum_{\pi \in \mathcal{P}(k)} \tau_N^e[\pi].
\end{align}
Note that, because of the partition $\pi$, the same random variables get merged together in $\tau^{e}_{N}[\pi],$ and their multiplicity appear in their exponent. For instance, suppose $k=5, N=10,$ and $\pi=\{\{1, 3\}, \{5\}, \{2, 4\}\}.$ Then $i_{1}=i_{2}, i_{2}=i_{4}$ and $i_{5}$ is neither of $i_{1}, i_{2}$, which yields a sequence of the form $\mathbb{E}[a_{19}a_{91}a_{19}a_{97}a_{71}].$ This sequence will reduce to $\mathbb{E}[a_{19}^{2}a_{91}]\mathbb{E}[a_{97}]\mathbb{E}[a_{71}].$ This enables us to analyze the limiting behavior of $\Tr(A^{k})/N$ in terms of $C_{k,l}s$ --- the limiting mixed moments. 

Now, for each partition $\pi$ of $\{1,2,\dots,k\},$ we consider a directed graph
$T_\pi = (V,E)$, where $V=\{V_1,V_2,\dots\}$ is the set of blocks of $\pi$ and $E$ is the multiset of directed edges defined as follows. There is an edge from block the $V_i$ to the block $V_j$ if there exists $m\in\{1,2,\dots,k\}$ such that $m\in V_i$ and $m+1\in V_j,$ where $k+1$ is identified with $1.$
For such a graph $T_\pi$, we can write
\begin{align}\label{eq:tauN_def}
\tau_N^e[\pi] = \tau_N^e[T_\pi]
= \frac{1}{N}\sum_{\substack{\phi:V\to [N]\\ \text{injective}}}
\prod_{(u,v)\in E} a_{\phi(u)\phi(v)}.
\end{align}

\begin{defn}[Thick tree]
A directed graph $T=(V, E)$ is called a thick tree if it is
a connected graph without loops and does not have any vertex pair $\{u, v\}$, which is connected \textcolor{blue}{only} by a single directed edge. Additionally, the underlying simple graph $\mathcal{T}=(V, \mathcal{E}),$ obtained by ignoring orientation and multiplicity satisfies $|V|=|\mathcal{E}|+1$.
\end{defn}

\begin{figure}[htbp]
\centering

\begin{minipage}{0.45\textwidth}
\centering

\begin{tikzpicture}[->, >=stealth, thick]

% Nodes
\node[circle, fill, inner sep=2pt] (v1) at (0,0) {};
\node[circle, fill, inner sep=2pt] (v2) at (2,1) {};
\node[circle, fill, inner sep=2pt] (v3) at (4,0) {};
\node[circle, fill, inner sep=2pt] (v4) at (2,-1) {};

% Valid edges (BLUE)
\draw[brown, bend left=20] (v1) to (v2);
\draw[brown, bend left=40] (v1) to (v2);

\draw[brown, bend left=20] (v2) to (v3);
\draw[brown, bend left=25] (v3) to (v2);

\draw[brown, bend left=20] (v1) to (v4);
\draw[brown, bend left=40] (v1) to (v4);

\end{tikzpicture}

\vspace{3mm}
\textbf{thick tree}

\end{minipage}
\hfill
\begin{minipage}{0.45\textwidth}
\centering

\begin{tikzpicture}[->, >=stealth, thick]

% Nodes
\node[circle, fill, inner sep=2pt] (v1) at (0,0) {};
\node[circle, fill, inner sep=2pt] (v2) at (1.8,1.2) {};
\node[circle, fill, inner sep=2pt] (v3) at (3.6,0) {};

% Violating edges (RED)

% single connection issue
\draw[brown, bend left=20] (v1) to (v2);
\draw[brown, bend left=40] (v2) to (v1);

% loop (not allowed)
\draw[green] (v2) edge[loop above] ();

% single edge between pair (violation)
\draw[green] (v2) -- (v3);

\end{tikzpicture}

\vspace{3mm}
\textbf{Not thick tree}

\end{minipage}
\caption{Examples of graphs satisfying the conditions of thick tree (left) and a graph violating them (right). The green-colored edges highlight the violations of the defining conditions.}
\label{fig:our_graph_comparison}

\end{figure}
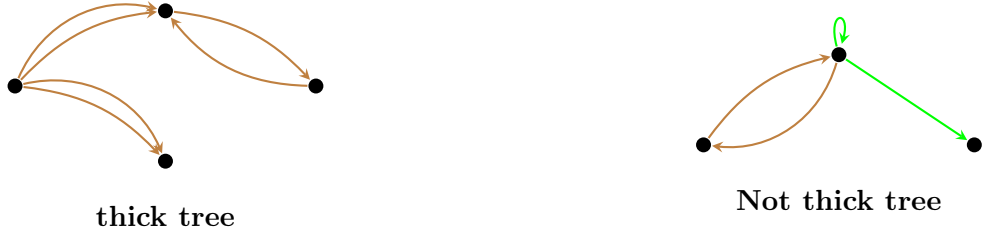

\begin{defn}[Fat Tree]
A directed graph $T=(V, E)$ is called a fat tree if it is a connected graph without loops and all the edges have multiplicity more than one. Additionally, the underlying simple graph $\mathcal{T}=(V, \mathcal{E}),$ obtained by ignoring orientation and multiplicity satisfies $|V|=|\mathcal{E}|+1$.
\end{defn}

% \textcolor{red}{
% \begin{defn}[\textcolor{blue}{"Our graph 2"}]
% A directed graph $T=(V, E)$ is called \textcolor{blue}{"Our graph 2"} if it is connected and there does not exist an ordered vertex pair $(u, v)$ such that they are connected by exactly one edge, or such that there is one directed edge in one direction and one or more edges in the opposite direction. Moreover, no vertex has a loop of multiplicity one \textcolor{orange}{I think it should not have loops at all because the graph is connected and the underlying simple graph must satisfy $|V|=|E|+1$}. The underlying simple graph, obtained by ignoring orientation and multiplicity, satisfies $|V|=|E|+1$. 
% \end{defn}
% }
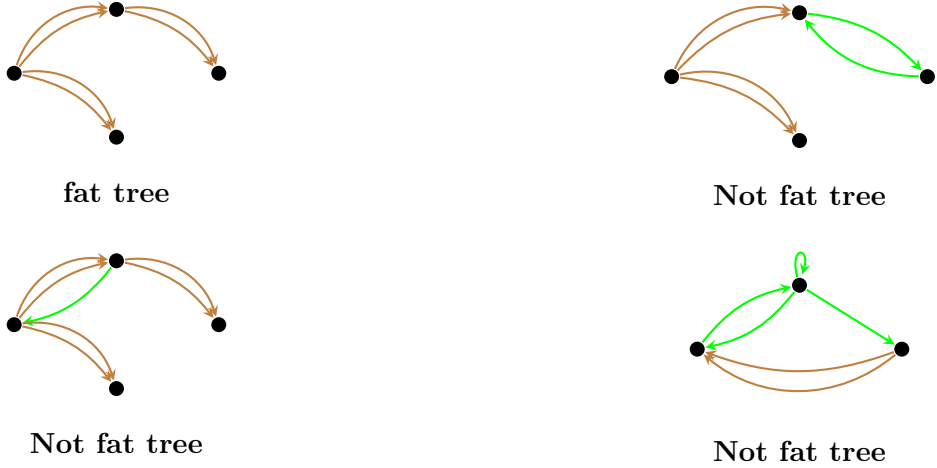
\begin{figure}[htbp]
\centering

% -------- First row --------
\begin{minipage}{0.45\textwidth}
\centering

\begin{tikzpicture}[->, >=stealth, thick, scale=0.85]

% Nodes
\node[circle, fill, inner sep=2pt] (v1) at (0,0) {};
\node[circle, fill, inner sep=2pt] (v2) at (1.6,1) {};
\node[circle, fill, inner sep=2pt] (v3) at (3.2,0) {};
\node[circle, fill, inner sep=2pt] (v4) at (1.6,-1) {};

% Valid edges (BLUE)
\draw[brown, bend left=20] (v1) to (v2);
\draw[brown, bend left=40] (v1) to (v2);

\draw[brown, bend left=20] (v2) to (v3);
\draw[brown, bend left=40] (v2) to (v3);

\draw[brown, bend left=20] (v1) to (v4);
\draw[brown, bend left=40] (v1) to (v4);

\end{tikzpicture}

\vspace{3mm}
\textbf{fat tree}

\end{minipage}
\hfill
\begin{minipage}{0.45\textwidth}
\centering

\begin{tikzpicture}[->, >=stealth, thick, scale=0.85]

% Nodes
\node[circle, fill, inner sep=2pt] (v1) at (0,0) {};
\node[circle, fill, inner sep=2pt] (v2) at (2,1) {};
\node[circle, fill, inner sep=2pt] (v3) at (4,0) {};
\node[circle, fill, inner sep=2pt] (v4) at (2,-1) {};

% Violations (RED)
\draw[brown, bend left=20] (v1) to (v2);
\draw[brown, bend left=40] (v1) to (v2);

\draw[green, bend left=20] (v2) to (v3);
\draw[green, bend left=25] (v3) to (v2); % mixed direction violation

\draw[brown, bend left=20] (v1) to (v4);
\draw[brown, bend left=40] (v1) to (v4);

\end{tikzpicture}

\vspace{3mm}
\textbf{Not fat tree}

\end{minipage}

\vspace{5mm}

% -------- Second row --------
\begin{minipage}{0.45\textwidth}
\centering

\begin{tikzpicture}[->, >=stealth, thick, scale=0.85]

% Nodes
\node[circle, fill, inner sep=2pt] (v1) at (0,0) {};
\node[circle, fill, inner sep=2pt] (v2) at (1.6,1) {};
\node[circle, fill, inner sep=2pt] (v3) at (3.2,0) {};
\node[circle, fill, inner sep=2pt] (v4) at (1.6,-1) {};

% Violations (RED)
\draw[brown, bend left=20] (v1) to (v2);
\draw[green, bend left=20] (v2) to (v1); % opposite direction
\draw[brown, bend left=40] (v1) to (v2);

\draw[brown, bend left=20] (v2) to (v3);
\draw[brown, bend left=40] (v2) to (v3);

\draw[brown, bend left=20] (v1) to (v4);
\draw[brown, bend left=40] (v1) to (v4);

\end{tikzpicture}

\vspace{3mm}
\textbf{Not fat tree}

\end{minipage}
\hfill
\begin{minipage}{0.45\textwidth}
\centering

\begin{tikzpicture}[->, >=stealth, thick, scale=0.85]

% Nodes
\node[circle, fill, inner sep=2pt] (v1) at (0,0) {};
\node[circle, fill, inner sep=2pt] (v2) at (1.6,1) {};
\node[circle, fill, inner sep=2pt] (v3) at (3.2,0) {};

% Violations (RED)
\draw[green, bend left=20] (v1) to (v2);
\draw[green, bend left=20] (v2) to (v1);

\draw[brown, bend left=20] (v3) to (v1);
\draw[brown, bend left=40] (v3) to (v1);

\draw[green] (v2) -- (v3); % single edge violation
\draw[green] (v2) edge[loop above] (); % loop violation

\end{tikzpicture}

\vspace{3mm}
\textbf{Not fat tree}

\end{minipage}

\caption{Examples of graphs satisfying the conditions of fat tree (top-left) and graphs that do not. The green-colored edges highlight the violations of the defining conditions.}
\label{fig:our_graph2_comparison}

\end{figure}

\begin{Thm}[Convergence of normalized trace]\label{Thm:Convergence of normalized trace}
    Let $A$ be the matrix same as in Condition \ref{cond:matrixcond_1}. 
For any partition
$\pi\in\bigcup_{k\ge1}\mathcal P(k)$, let $\tau_N^e[T_\pi]$ be as defined in \eqref{eq:tauN_def}. Then, as $N\to\infty$, the following hold.

\begin{enumerate}[label=(\alph*)]
    \item For $\alpha > 1$ we have
    $$
    \mathbb{E}[\tau_N^e[T_\pi]] \longrightarrow 0 .
    $$
    \item   For $\alpha=1$ we have
    \begin{align*}
        &\mathbb{E}\big[\tau_N^e[T_\pi]\big]\longrightarrow
    \tau^e[T_\pi]
   :=
    \begin{cases}
   \displaystyle\prod_{\substack{k,l\ge0\\ k+l\ge2}} C_{k,l}^{\,q_{k,l}}, &\text{if $T_\pi$ is a thick fat},\\[0.8em]
    0, & \text{otherwise,}
   \end{cases}
    \end{align*}
    where $q_{k,l}$ denote the number of ordered vertex pairs $(i,j)$ with $i<j$ such that there are exactly $k$ edges from vertex $i$ to vertex $j$ and $l$ edges from vertex $j$ to vertex $i$ in thick tree $T_\pi$.
   \item For $\alpha<1$ we have 
    \begin{align*}
\mathbb{E}\big[\tau_N^e[T_\pi]\big] \sim
\begin{cases}
0, & \parbox[t]{0.6\textwidth}{\setstretch{0.9}\raggedright
$T_\pi$ has either a vertex pair $\{u, v\}$ connected by a single directed edge, or a vertex with a single loop (or both)},\\[0.6em]
O(N^{|V|-1-\alpha \mathfrak{p}}), & \text{otherwise,}
\end{cases}
\end{align*}
where $|V|$ is the number of vertices of $T_\pi$, and $\mathfrak{p}$ is the number of edges of the graph, which is obtained by forgetting the multiplicity and orientation of the edges of the graph $T_\pi$.
\end{enumerate}
\end{Thm}
\begin{proof}
For a fixed partition $\pi$ in $\bigcup_{k\ge1}\mathcal P(k),$ let $\tau_N^e[T_\pi]$ be as defined in \eqref{eq:tauN_def}. Then
\begin{align}\label{eq: tauN}
\mathbb{E}\big[\tau_N^e[T_\pi]\big]
&=\frac{1}{N}\sum_{\substack{\phi:V\to [N]\\ \text{injective}}}
\mathbb{E}\bigg[\prod_{(u,v)\in E} a_{\phi(u)\phi(v)}\bigg].
\end{align}
Since the distribution of the entries $a_{ij}$ is invariant under relabeling of the
indices, 
$$
\mathbb{E}\bigg[\prod_{(u,v)\in E} a_{\phi(u)\phi(v)}\bigg]
$$
does not depend on the particular choice of the injective map $\phi$. Therefore,
\begin{align}\label{eq: tauN}
\mathbb{E}\big[\tau_N^e[T_\pi]\big]
%&=\frac{1}{N}\sum_{\substack{\phi:V\to [N]\\ \text{injective}}}
%\mathbb{E}\bigg[\prod_{(u,v)\in E} a_{\phi(u)\phi(v)}\bigg]\notag\\
&=\frac{1}{N}\frac{N!}{(N-|V|)!}\mathbb{E}\bigg[\prod_{(u,v)\in E} a_{\phi(u)\phi(v)}\bigg]\notag\\
    &=:\frac{(N-1)!}{(N-|V|)!}\delta_N^0(T(a_N)),
\end{align}
where $\delta_N^0(T(a_N))=\mathbb{E}\bigg[\prod_{(u,v)\in E} a_{\phi(u)\phi(v)}\bigg].$ Since $\mathbb{E}[x_{ij}] = 0$ for $1 \leq i,j \leq N$, we have
\begin{equation}\label{eq: deltaN is zero}
    \delta_N^0(T(\mathbf{a}_N)) = 0
\end{equation}
 whenever there exists a vertex pair $\{u, v\}$, which is connected only by a single edge in $T_\pi$, or there is a vertex having a loop of multiplicity one. Therefore, to obtain a non-zero contribution, from this point onward, we assume that $T_\pi$ does not have any vertex pair $\{u, v\}$, which is only connected by a single edge, and there is no vertex having a loop of multiplicity one. We denote such graphs by $\hat{T}_{\pi}$. The reduced simple graph will be denoted by $\hat{\mathcal{T}}_{\pi},$ which is obtained after ignoring the direction and multiplicity of the edges of $\hat{T}_{\pi}.$
% \textcolor{blue}{but $\mathbb{E}[x_{ii}]=0$ and $\mathbb{E}[|x_{ii}|/N^{1/2-1}]=O(1)$ from ,
% $$
% %\frac{\mathbb{E}[x_{ii}^{k}]}{N^{k/2-\alpha}} = O(1), \qquad 
% \frac{\mathbb{E}[|x_{ii}|^{k}]}{N^{k/2-\alpha}} = O(1), \text{ for all } k \ge 1
% $$}

By Condition \ref{cond:matrixcond_1}, the diagonal entries are independent among themselves and of off-diagonal entries as well. The off-diagonal entries' dependence occurs only within the pairs $(x_{ij},x_{ji}),$ while $(x_{ij},x_{ji})$ is independent of $(x_{kl},x_{lk})$ if $(i, j)\neq (k, l).$ 

For $k \ge 2$, let $p_k$ denote the number of vertices carrying exactly $k$ loops. For $k, l \ge 0$ with $k+l\geq 2,$ let $q_{k,l}$ denote the number of ordered vertex pairs $(i,j)$ with $i<j$ such that there are exactly $k$ edges from vertex $i$ to vertex $j$ and $l$ edges from vertex $j$ to vertex $i$.
Then 
\begin{align}
    \delta_N^0(T(a_N))
    &=\mathbb{E}\bigg[\prod_{(u,v)\in E} a_{\phi(u)\phi(v)}\bigg]\notag\\
    &=\prod_{k\geq2}\big(\mathbb{E}[a_{11}^k]\big)^{p_k}\displaystyle\prod_{\substack{k,l\ge0\\ k+l\ge2}}\big(\mathbb{E}[a_{12}^{k}a_{21}^{l}]\big)^{q_{k,l}}.\notag
\end{align}
Recall that $a_{ij}=x_{ij}/\sqrt{N}$. By Condition \ref{cond:matrixcond_1}, we have
$$
N^{-\frac{k}{2}+\alpha }\mathbb{E}[|x_{11}|^k]=O(1),
\qquad
N^{-\frac{k+l}{2}+\alpha }\mathbb{E}[x_{12}^k x_{21}^l]
\longrightarrow C_{k,l}.
$$
Therefore, each loop contributes a factor of order $N^{-\alpha}$, and each ordered vertex
pair $(i,j),$ with $i<j$ also contributes a factor of order $N^{-\alpha}$.
Therefore,
\begin{align}\label{eq: deltaN}
    &\delta_N^0(T(a_N))
=
N^{-\alpha \mathfrak{p}}
\left(
\prod_{k\ge2} \left(N^{-\frac{k}{2}+\alpha }\mathbb{E}[x_{11}^k]\right)^{p_k}
\right)
\left(
\displaystyle\prod_{\substack{k,l\ge0\\ k+l\ge2}}
\Big(
N^{-\frac{k+l}{2}+\alpha }
\mathbb{E}[x_{12}^k x_{21}^l]
\Big)^{q_{k,l}}
\right),
\end{align}
where
$$
\mathfrak{p}
:=
\sum_{k\ge2} p_k
+
\displaystyle\sum_{\substack{k,l\ge0\\ k+l\ge2}} q_{k,l}.
$$
Here $\mathfrak{p}$ is the number of edges of $\hat{\mathcal{T}}_{\pi}$. Now, by using \eqref{eq: deltaN} in \eqref{eq: tauN} we have
\begin{align}\label{eq: tauN_Final}
\mathbb{E}\big[\tau_N^e[\hat{T}_{\pi}]\big]
    &= \frac{(N-1)!}{(N-|V|)!}\frac{1}{N^{\alpha \mathfrak{p}}}\left(
\prod_{k\ge2} \left(N^{-\frac{k}{2}+\alpha }\mathbb{E}[x_{11}^k]\right)^{p_k}
\right)
\left(
\displaystyle\prod_{\substack{k,l\ge0\\ k+l\ge2}}
\Big(
N^{-\frac{k+l}{2}+\alpha }
\mathbb{E}[x_{12}^k x_{21}^l]
\Big)^{q_{k,l}}
\right)\notag\\
 &=N^{|V|-(\alpha \mathfrak{p}+1)}\left(
\prod_{k\ge2} \left(N^{-\frac{k}{2}+\alpha }\mathbb{E}[x_{11}^k]\right)^{p_k}
\right)
\displaystyle\prod_{\substack{k,l\ge0\\ k+l\ge2}}
\Big(
N^{-\frac{k+l}{2}+\alpha }
\mathbb{E}[x_{12}^k x_{21}^l]
\Big)^{q_{k,l}}
+o(1).
\end{align}
Since 
for all $k,l \ge 0$ with $k+l \geq 2$
% $$
%  \frac{\int z_1^kz_2^ld\mu_{X}(z)}{N^{(k+l)/2-1}} \longrightarrow C_{k,l}, \quad \frac{\int |z|^kd\nu_{X}(z)}{N^{k/2-1}} = O(1)
% $$ 
$$
N^{-\frac{k+l}{2}+\alpha }\mathbb{E}[x_{12}^k x_{21}^l]
\longrightarrow C_{k,l}, 
$$
and for all $k\ge 2$ 
$$
N^{-\frac{k}{2}+\alpha }\mathbb{E}[|x_{11}|^k] = O(1),
$$
by Condition \ref{cond:matrixcond_1}, the asymptotic behavior of \eqref{eq: tauN_Final}
depends on the power of $N$ that is on $N^{|V|-1-\alpha \mathfrak{p}}.$ We now discuss the following cases for different values of $\alpha$ as follows.

\textbf{Case (a) $(\alpha>1)$.} For any connected graph, we have $|V| \leq \mathfrak{p} + 1$. Since $\hat{T}_\pi$ is connected by definition, it follows that $|V| \leq \mathfrak{p} + 1$. Since $\alpha > 1$, we have
$$
|V|-1-\alpha \mathfrak{p} \le \mathfrak{p}-\alpha \mathfrak{p} = \big(1-\alpha\big)\mathfrak{p} < 0,
$$
which implies
$$
\mathbb{E}\big[\tau_N^e[\hat{T}_\pi]\big]\longrightarrow 0.
$$
If $T_\pi$ has a vertex pair $\{u, v\}$, which is connected only by a single edge, or a vertex with a single loop (or both), then from \eqref{eq: deltaN is zero} we have 
$\mathbb{E}\big[\tau_N^e[T_\pi]\big]= 0.$
Therefore, for any graph $T_\pi$ associated to $\pi$, we have 
$
\mathbb{E}\big[\tau_N^e[T_\pi]\big]\longrightarrow 0.
$

\textbf{ Case (b) $(\alpha=1)$.} In this case, we have 
$$
|V|-\alpha \mathfrak{p}-1= |V|-\mathfrak{p}-1,
$$ 
where $\mathfrak{p}$ is the number of edges of $\hat{\mathcal{T}_{\pi}}$. From \cite[Theorem 2.1.4]{west2001introduction}, $|V|= \mathfrak{p}+1$ if and only if $T_\pi$ is thick tree.
Thus, for thick tree $T_\pi$, we have
$$
\mathbb{E}\big[\tau_N^e[T_\pi]\big]
\longrightarrow
\displaystyle\prod_{\substack{k,l\ge1\\ k+l\ge2}} C_{k,l}^{\,q_{k,l}}.
$$

Now, suppose that $T_{\pi}$ is not thick tree; that is, there exists a vertex pair $\{u, v\}$, which is connected by a single directed edge only, or there is a vertex having a loop of multiplicity one, or $|V|\neq \mathfrak{p} + 1$. In the first two cases, we have $\mathbb{E}\big[\tau_N^e[T_\pi]\big]
= 0.$ In the last case, since $|V| \leq \mathfrak{p} + 1$ for any connected graph, it follows that $|V|-\mathfrak{p}-1<0,$ which implies
$$
\mathbb{E}\big[\tau_N^e[T_\pi]\big]
\longrightarrow 0.
$$

\textbf{Case (c) $(\alpha<1)$.}  
If $T_\pi$ has a vertex pair $\{u, v\}$, which is connected by a single directed edge only, or there exists a vertex with a single loop, then from \eqref{eq: deltaN is zero} we have 
$
\mathbb{E}\big[\tau_N^e[T_\pi]\big]
= 0
$

Since $\alpha<1$, for the graph  $\hat{T}_\pi$, we have
$$
|V|-1-\alpha \mathfrak{p} \le \mathfrak{p}-\alpha \mathfrak{p} 
= (1-\alpha)\mathfrak{p}>0.
$$
This inequality does not give the sign of $|V|-1-\alpha \mathfrak{p}$ in general, it may be positive, zero, or negative.
Thus
$$
\mathbb{E}\big[\tau_N^e[T_\pi]\big] \sim O(N^{|V|-1-\alpha \mathfrak{p}}).
$$
\end{proof}
Going ahead, we now discuss the CLT for the normalized trace in the following subsection.

\subsection{CLT for the normalized traces}

We consider the centered and normalized random variables
\begin{align}\label{eq:centered and normalized random variables Z_N}
    &Z_N(T_\pi)
:= \sqrt{N}\,\big(\tau_N^e[T_\pi] - \mathbb{E}[\tau_N^e[T_\pi]]\big).
\end{align}
It is sufficient to prove the convergence of $\{Z_N(T_\pi)\}_{\pi\in\bigcup_{k\ge1}\mathcal{P}(k) }$ to Gaussian process in order to establish the CLT for $\{Z_N(k)\}_{k \geq 1}$,
where
$$
Z_N(k)
= \sqrt{N}\left(\frac{1}{N}\operatorname{Tr}(A^k)
- \mathbb{E}\left[\frac{1}{N}\operatorname{Tr}(A^k)\right]\right).
$$
\begin{Thm}[CLT for the normalized traces]\label{thm:CLT for normalized traces}
Let $A$ be the matrix same as in Condition \ref{cond:matrixcond_1} with $\alpha=1$. For any partition $\pi\in\bigcup_{k}\mathcal{P}(k)$, let $Z_N(T_\pi)$ be defined as in \eqref{eq:centered and normalized random variables Z_N}.
% $$
% Z_N(T_\pi)
% =
% \sqrt{N}\Big(
% \tau_N^0[T_\pi]-\mathbb{E}[\tau_N^0[T_\pi]]
% \Big).
% $$
Then the family of random variables
$
\big(Z_N(T_\pi)\big)_{\pi\in\bigcup_{k}\mathcal{P}(k)}
$
converges to a centered
Gaussian process
$
\big(z(T_\pi)\big)_{\pi\in\bigcup_{k}\mathcal{P}(k)}
$ as $N\to\infty.$ 
Moreover, for any two partitions $\pi_1,\pi_2$,
the covariance of the limiting process is given by
$$
\mathrm{Cov}\left(z(T_{\pi_1}),z(T_{\pi_2})\right)
=
\sum_{T\in\mathcal{P}^{\#}(T_{\pi_1},T_{\pi_2})}
\tau^e[T],
$$
where $\mathcal{P}^{\#}(T_{\pi_1},T_{\pi_2})$ denotes the collection of graphs obtained by disjoint copies of the graphs $T_{\pi_1}$ and $T_{\pi_2}$ and gluing them with at least one common edge, and $\tau^e[T]$ is same as defined in the Theorem \ref{Thm:Convergence of normalized trace}(b).
\end{Thm}
\begin{proof}
Let $\pi_1,\pi_2, \dots, \pi_r \in \bigcup_{k} \mathcal{P}(k)$. 
Let $T_1,T_2, \dots, T_r$ be directed graphs, each corresponding to a partition $\pi_j$, with $T_j = (V_j, E_j)$. These graphs are defined in the same way as in Section \ref{section:Central Limit Theorem for Generalized Moments for corr matrices}. Recall from \eqref{eq:centered and normalized random variables Z_N} 
$$
Z_N(T_i)
=
\sqrt{N}\Big(
\tau_N^e[T_i]
-
\mathbb{E}[\tau_N^e[T_i]]
\Big),
$$
where 
$$
\tau_N^e[T_i]
=
\frac{1}{N}
\sum_{\phi:V_i\to[N]\atop \text{injective}}
\prod_{(u,v)\in E_i} a_{\phi(u)\phi(v)}.
$$

To prove the convergence of $\{Z_N(T_j)\}$ to a Gaussian process, 
it is sufficient to show that the joint moments
$$
\mathbb{E}\left[\prod_{j=1}^r Z_N(T_{j})\right]
$$
satisfy Wick’s formula. We now express
\begin{align*}
   & \mathbb{E}\left[\prod_{j=1}^r Z_N(T_j)\right]\\
&=
N^{-r/2}
\sum_{\substack{\phi_1,\phi_2,\dots,\phi_r \\ \phi_j:V_j \rightarrow [N] \text{ injective}}}
\mathbb{E}\left[
\prod_{j=1}^r
\Bigg(
\prod_{(u,v)\in E_{j}} a_{\phi_{j(u)}\phi_{j(v)}}
-
\mathbb{E}\bigg[\prod_{(u,v)\in E_{j}} a_{\phi_{j(u)}\phi_{j(v)}}\bigg]
\Bigg)
\right].
\end{align*}

Let $\mathcal{P}(V_1,V_2,\dots,V_r)$ be the set of partitions of disjoint union of $V_1\cup V_2\cdots\cup V_r$
such that each block contains at most one vertex from each $V_j$. For $\sigma\in\mathcal{P}(V_1,V_2,\dots,V_r)$,
let
$$
S_\sigma =
\left\{
(\phi_1,\phi_2,\dots,\phi_r):
\phi_i(v)=\phi_j(v')
\iff v\sim_\sigma v'
\right\}.
$$
Then
$$
\mathbb{E}\left[\prod_{j=1}^r Z_N(T_j)\right]
=
\sum_{\sigma\in\mathcal{P}(V_1,V_2,\dots,V_r)}
N^{-r/2} \sum_{(\phi_1,\phi_2,\dots,\phi_r)\in S_\sigma}
\omega_N(\phi_1,\phi_2,\dots,\phi_r),
$$
where $$\omega_N(\phi_1,\phi_2,\dots,\phi_r)=\mathbb{E}\left[
\prod_{j=1}^r
\Bigg(
\prod_{(u,v)\in E_{j}} a_{\phi_{j(u)}\phi_{j(v)}}
-
\mathbb{E}\bigg[\prod_{(u,v)\in E_{j}} a_{\phi_{j(u)}\phi_{j(v)}}\bigg]
\Bigg)
\right].$$
Since $\omega_N(\phi_1,\phi_2,\dots,\phi_r)$ depends only on $\sigma$ so we denote
$\omega_N(\phi_1,\phi_2,\dots,\phi_r)=\omega_N(\sigma).$ 
Thus,
\begin{align}\label{eq:exp_product_2}
\mathbb{E}\left[\prod_{j=1}^r Z_N(T_j)\right]
&=
\sum_{\sigma\in\mathcal{P}(V_1,V_2,\dots,V_r)}
N^{-r/2}
\sum_{(\phi_1,\phi_2,\dots,\phi_r)\in S_\sigma}
\omega_N(\phi_1,\phi_2,\dots,\phi_r)\\
&=\sum_{\sigma\in\mathcal{P}(V_1,V_2,\dots,V_r)}
   N^{-r/2}
   \frac{N!}{(N - |\sigma|)!}
   \omega_N(\sigma).\notag
\end{align}
Now, if any graph $ T_i$ does not share an edge with any of the other graphs under $\sigma$, then by centering and independence of the entries of $ A $, we have $\omega_N(\sigma) = 0.$
Only partitions $\sigma$ that identify at least one edge between two graphs contribute. Therefore, we define
\begin{align*}
&\mathcal{P}^{\#}(V_1,V_2,\dots,V_r)\\
&= \Big\{\sigma \in \mathcal{P}(V_1,V_2,\dots,V_r)\ \Big|\ 
\parbox{0.6\textwidth}{\setstretch{0.9}
each graph $T_i$ shares at least one edge with another graph $T_{j}$ for $j\neq i$
}
\Big\}.
\end{align*}

Hence, from \eqref{eq:exp_product_2} we have
\begin{align}\label{eq:exp_product_3}
\mathbb{E}\left[\prod_{j=1}^r Z_N(T_j)\right]
&=\sum_{\sigma\in\mathcal{P}^{\#}(V_1,V_2,\dots,V_r)}
   N^{-r/2}
   \frac{N!}{(N - |\sigma|)!}
   \omega_N(\sigma)\notag\\
   &=\sum_{\sigma\in \mathcal{P}^{\#}(V_1,V_2,\dots,V_r)}
 N^{-r/2+|\sigma|}\,\omega_N(\sigma)\,(1+O(N^{-1})).\notag
\end{align}

We now expand $\omega_{N}(\sigma)$ as
\begin{align}
    &\omega_N(\sigma)
= \sum_{B\subseteq\{1,2,\dots,r\}}
(-1)^{r-|B|}
\mathbb{E}\bigg[
  \prod_{j\in B}
  \prod_{(u,v)\in E_j}a_{\phi_{j(u)}\phi_{j(v)}}
\bigg]
\prod_{j\notin B}
\mathbb{E}\bigg[
  \prod_{(u,v)\in E_j}a_{\phi_{j(u)}\phi_{j(v)}}
\bigg].
\end{align}
Since $\mathbb{E}[a_{ij}] = 0$ for $1\leq i,j\leq N$, we have 
\begin{equation}\label{eq:omegaN is zero}
    \omega_N(\sigma) = 0
\end{equation}
 whenever there exists a vertex pair $\{u, v\}$, which is connected by a single directed edge only, or there is a vertex having a loop of multiplicity one (or both), in $T_\pi$. Therefore, to obtain a non-zero contribution, from this point onward, we assume that $T_\pi$ does not have any vertex pair $\{u, v\}$, which is connected by a single directed edge, and there is no vertex having a loop of multiplicity one. We denote such graphs by $\hat{T}_{\pi}$. The reduced simple graph is denoted by $\hat{\mathcal{T}}_{\pi},$ which is obtained after ignoring the direction and multiplicity of the edges of $\hat{T}_{\pi}$.

For a fixed $B\subseteq\{1,2,\dots,r\}$, let $T_B$ be the directed graph obtained by merging
the vertices of $T_j$, $j\in B$, that belong to the same block of $\sigma$. Let $p_k$ denote the number of vertices of $\hat{T}_B$ with $k$ loops attached,
and $q_{k,l}$ the number of ordered vertex pairs $(u,v)$ with $u<v$
such that there are $k$ edges from $u$ to $v$ and $l$ edges from $v$ to $u$. By independence and the correlation assumptions on $A$, we have
\begin{align}
   &\mathbb{E}\left[\prod_{j\in B}\prod_{(u,v)\in E_j} a_{\phi_{j(u)}\phi_{j(v)}}\right]\notag\\
% &= N^{-|\bar E_B|}
% \prod_{k\ge1}\left(\int t^k d\mu_N(t)\,N^{-k/2+1}\right)^{p_k}
% \prod_{k,l\ge0}\left(\int t^{k+l}\,d\mu_N(t)\,N^{-(k+l)/2+1}\right)^{q_{k,l}}\notag\\
&= \prod_{k\geq2}\big(\mathbb{E}[a_{11}^k]\big)^{p_k}\displaystyle\prod_{\substack{k,l\ge0\\ k+l\ge2}}\big(\mathbb{E}[a_{12}^{k}a_{21}^{l}]\big)^{q_{k,l}}\notag\\
&= N^{-\mathfrak{p}_B}
\prod_{k\geq2}\bigg(\frac{\mathbb{E}[|x_{11}|^k]}{N^{k/2-1}}\bigg)^{p_k}\displaystyle\prod_{\substack{k,l\ge0\\ k+l\ge2}}\bigg(\frac{\mathbb{E}[x_{12}^{k}x_{21}^{l}]}{N^{(k+l)/2-1}}\bigg)^{q_{k,l}}\notag\\
&=: N^{- \mathfrak{p}_B}\,\delta_N(B),
\end{align}
where $$
\mathfrak{p}_B:=
\sum_{k\ge2} p_k
+
\displaystyle\sum_{\substack{k,l\ge0\\ k+l\ge2}} q_{k,l}.
$$
Here $\mathfrak{p}_B$ is the number of edges of $\hat{\mathcal{T}}_B$.

Let $T_{B^{\complement}}$ be the graph obtained by merging the vertices of the graphs $T_j$, $j\notin B$,
that belong to the same block of $\sigma$. Let $p_{k}'$ denote the number of vertices of $\hat{T}_{B^{\complement}}$ with $k$ loops attached.
Let $q_{k,l}'$ denotes the number of ordered vertex pair $(u,v)$ with $u<v$
such that there are $k$ directed edges from $u$ to $v$ and $l$ directed edges from $u$ to $v.$ By independence and the correlation assumptions on $A$, we have
\begin{align*}
    &\prod_{j\notin B}
\mathbb{E}
\bigg[
\prod_{(u,v)\in E_j}
a_{\phi_{j(u)}\phi_{j(v)}}
\bigg]\\
&=\prod_{k\geq2}\big(\mathbb{E}[a_{11}^k]\big)^{p_{k}{'}}\displaystyle\prod_{\substack{k,l\ge0\\ k+l\ge2}}\big(\mathbb{E}[a_{12}^{k}a_{21}^{l}]\big)^{q_{k,l}'}\notag\\
&= N^{-(\mathfrak{p}_{B^\complement})}
\prod_{k\geq2}\bigg(\frac{\mathbb{E}[|x_{11}|^k]}{N^{k/2-1}}\bigg)^{p_{k}{'}}\displaystyle\prod_{\substack{k,l\ge0\\ k+l\ge2}}\bigg(\frac{\mathbb{E}[x_{12}^{k}x_{21}^{l}]}{N^{(k+l)/2-1}}\bigg)^{q_{k,l}'}\notag\\
&=: N^{- (\mathfrak{p}_{B^\complement})}\,\delta_N(B^\complement),
\end{align*}
where $$
\mathfrak{p}_{B^\complement}:=
\sum_{k\ge1} p_k{'}
+
\displaystyle\sum_{\substack{k,l\ge0\\ k+l\ge2}} q_{k,l}'.
$$
Here $\mathfrak{p}_{B^{\complement}}$ is the number of edges of $\hat{\mathcal{T}}_{B^{\complement}}$.
Let $T_\sigma$ be the directed graph obtained by merging all the vertices of
$T_1,T_2,\dots,T_r$ belonging to the same block of $\sigma$.
Let $c_\sigma$ be the number of connected components of $\hat{T}_\sigma$ and $\mathfrak{p}_\sigma$ is the number of edges of $\hat{\mathcal{T}}_\sigma$. Then

\begin{align}\label{eq: exp of product last}
    &\mathbb{E}\left[\prod_{j=1}^r Z_N(T_j)\right]\\
    &=\sum_{\sigma\in\mathcal{P}^\#(V_1,V_2,\dots,V_r)}
\sum_{B\subseteq\{1,2,\dots,r\}}N^{-r/2+|\sigma|- \mathfrak{p}_{B}-\mathfrak{p}_{B^\complement}}(-1)^{r-|B|}
\delta_N(B)\delta_N(B^\complement)[1+O(N^{-1})]\notag\\
    &=\sum_{\sigma\in\mathcal{P}^\#(V_1,V_2,\dots,V_r)}N^{c_\sigma - r/2}N^{|\sigma|-c_\sigma-\mathfrak{p}_\sigma}
\sum_{B\subseteq\{1,2,\dots,r\}}
N^{\mathfrak{p}_\sigma-\mathfrak{p}_{B}-\mathfrak{p}_{B^\complement}}
(-1)^{r-|B|}\notag\\
&\;\;\; \times\delta_N(B)\delta_N(B^\complement)[1+O(N^{-1})].\notag
\end{align}

Here the quantities $\delta_N(B)$ and $\delta_N(B^\complement)$ are bounded. Thus, the asymptotic behavior of \eqref{eq: exp of product last}
depends on the exponent of $N$ that is on 
$$
(c_\sigma - r/2)+(\mathfrak{p}_\sigma-\mathfrak{p}_{B}-\mathfrak{p}_{B^\complement})+
(|\sigma|-c_\sigma-\mathfrak{p}_\sigma).
$$
We first consider $\sum_{B\subseteq\{1,2,\dots,r\}}
N^{\mathfrak{p}_\sigma-\mathfrak{p}_{B}-\mathfrak{p}_{B^\complement}}.$ Here $\mathfrak{p}_\sigma-\mathfrak{p}_{B}-\mathfrak{p}_{B^\complement} \le 0$ and equality holds iff $ B = \{1,2,\dots,r\} $ or $ B^\complement = \{1,2,\dots,r\} $. Without loss of generality, we assume $B = \{1,2,\dots,r\}.$ Then,
\begin{align}\label{eq:power of N 2nd eq}
    N^{\mathfrak{p}_\sigma-\mathfrak{p}_{B}-\mathfrak{p}_{B^\complement}}
= \mathbf{1}_{\{B=\{1,2,\dots,r\}\}} + O(N^{-1}).
\end{align}

We now consider the remaining factors $N^{c_\sigma - r/2}N^{|\sigma|-c_\sigma-\mathfrak{p}_\sigma}.$ Since $c_\sigma\leq r/2$ we have
$$
N^{c_\sigma-r/2} = O(N^{-\epsilon}) \text{ if }c_\sigma < r/2,
$$
where $\epsilon>0.$
Now, from \cite[Theorem 2.1.4]{west2001introduction}, $\mathfrak{p}_\sigma+c_\sigma-|\sigma|$ is the number of cycles of the graph $\hat{\mathcal{T}}_\sigma$. Thus $\mathfrak{p}_\sigma+c_\sigma-|\sigma|\geq0$ which implies $|\sigma|-c_\sigma-\mathfrak{p}_\sigma \leq 0.$ Equality holds iff $\hat{\mathcal{T}}_\sigma$ is a forest. Thus
\begin{align*}
    N^{|\sigma|-c_\sigma-\mathfrak{p}_\sigma}= O(N^{-1}) \text{ if }|\sigma|-c_\sigma-\mathfrak{p}_\sigma < 0
\end{align*}
Consequently, if either $ c_\sigma < r/2 $ or $\hat{\mathcal{T}}_\sigma$ is not a forest, then 
$$
N^{c_\sigma - r/2}
N^{|\sigma| - c_\sigma - \mathfrak{p}_\sigma}
= O(N^{-1}).
$$
Thus, to obtain a non zero contribution, we need $c_\sigma=r/2$, and graph $\hat{\mathcal{T}}_\sigma$ to be a forest, which is a collection of disjoint trees. Hence 
\begin{align}\label{eq:power of N 3rd eq}
   & N^{|\sigma|-c_\sigma-\mathfrak{p}_\sigma}
= \mathbf{1}_{\{ \hat{\mathcal{T}}_\sigma\text{ is a forest}\}} + O(N^{-1}),
\end{align}
and
\begin{align}\label{eq:csigma-r/2}
    N^{c_\sigma - r/2} =
\begin{cases}
1, & \text{if } c_\sigma=r/2\\[4pt]
O(N^{-1}), & \text{otherwise}.
\end{cases}
\end{align}
Each partition $\sigma\in \mathcal{P}^\#(V_1,V_2,\dots,V_r)$ induces another partition $\bar\sigma$ of $\{1,2,\dots,r\}$, defined by
$$
i \sim_{\bar\sigma} j
\quad \text{if and only if} \quad
T_i \text{ and } T_j \text{ belong to the same connected component of } \hat{T}_\sigma.
$$
Then we can write \eqref{eq:csigma-r/2} as follows
\begin{align}\label{eq: power of N 1st eq}
    N^{c_\sigma - r/2} =
\begin{cases}
1, & \text{if } \bar\sigma \in \mathcal{P}_2(r),\\[4pt]
O(N^{-1}), & \text{otherwise},
\end{cases}
\end{align}
where $\mathcal{P}_2(r)$ denotes the set of pair partition of $\{1,2,\dots,r\}.$

Thus from equations \eqref{eq: exp of product last},\eqref{eq:power of N 2nd eq}, \eqref{eq:power of N 3rd eq}, and \eqref{eq: power of N 1st eq} we have 
\begin{align}\label{eq: exp}
    &\mathbb{E}\left[\prod_{j=1}^r Z_N(T_j)\right]\notag\\
&=
\sum_{\pi\in\mathcal{P}_2(r)}
\sum_{\substack{\sigma\in\mathcal{P}^{\#}(V_1,\dots,V_r)\\ \bar\sigma=\pi}}
\mathbf{1}_{\{\hat{\mathcal{T}}_\sigma \text{ is a forest of } r/2 \text{ trees}\}}
\,\delta_N(\sigma)
+ O(N^{-1})\notag\\
&=\sum_{\pi\in\mathcal{P}_2(r)}
\prod_{\{i,j\}\in\pi}
\left(
\sum_{\sigma\in\mathcal{P}^{\#}(V_i,V_j)}
\mathbf{1}_{\{\hat{\mathcal{T}}_\sigma \text{ is a tree}\}}
\,\delta_N(\sigma)
\right)
+ O(N^{-1}),
\end{align}
where in the last equality we have used the independence of the entries of the matrix $ A$ to split $\delta_N(\sigma)$, for $\sigma\in \mathcal{P}^{\#}(V_1,\dots,V_r),$ as a product of $\delta_{N}(\sigma)$, where $\sigma\in \mathcal{P}^{\#}(V_{i}, V_{j})$.
We now evaluate the covariance kernel, let $r=2$
\begin{align*}
   & \mathbb{E}[Z_N(T_1)Z_N(T_2)]\notag\\
&=
\sum_{\sigma\in\mathcal{P}^{\#}(V_1,V_2)}
\mathbf{1}_{\{\hat{\mathcal{T}}_\sigma \text{ is a tree}\}}
\,\delta(\sigma)
+ o(1)\notag\\
&=\sum_{\hat{\mathcal{T}}_\sigma \in\mathcal{P}^{\#}(T_1,T_2)} \tau_e[\hat{\mathcal{T}}_\sigma ]+ o(1),
\end{align*}
where $\delta(\sigma) = \lim_{N\to\infty}\delta_N(\sigma)
= \tau^e[\hat{T}_\sigma].$ Here $\tau^e[\hat{T}_\sigma]$ is same as $\tau^e[T]$ defined in the Theorem \ref{Thm:Convergence of normalized trace}(b).

Thus, from equation \eqref{eq: exp} we have 
\begin{align}
    &\mathbb{E}\left[\prod_{j=1}^r Z_N(T_j)\right]\notag\\
&=
\sum_{\pi\in\mathcal{P}_2(r)}
\prod_{\{i,j\}\in\pi}
\left(
\sum_{\sigma\in\mathcal{P}^{\#}(V_i,V_j)}
\mathbf{1}_{\{\hat{\mathcal{T}}_\sigma \text{ is a tree}\}}
\,\delta_N(\sigma)
\right)
+ o(1)\notag\\
&=
\sum_{\pi\in\mathcal{P}_2(r)}
\prod_{\{i,j\}\in\pi}
\sum_{T\in\mathcal{P}^{\#}(T_i,T_j)} \mathbb{E}[Z_N(T_i)Z_N(T_j)]
+  o(1),
\end{align}
which is precisely Wick's formula. Hence, the proof is complete.
\end{proof}

%%%%%%%%%%%%%%%%%%%%%%%%%%%%%%%%%%%%%%%%%%%%%%%%%%%%%%%%%%%%%%%%%%%%%%%%%%%
%%%%%%%%%%%%%%%%%%    Non-Hermitian Matrices with exploding moments   %%%%%%%%%%%%%%%%%%%%%
%%%%%%%%%%%%%%%%%%%%%%%%%%%%%%%%%%%%%%%%%%%%%%%%%%%%%%%%%%%%%%%%%%%%%%%%%%%

\section{Non-Hermitian matrices}
In this section, we consider a non-Hermitian matrix model with
independent entries and exploding moments.
Unlike the previous Section \ref{section:Central Limit Theorem for Generalized Moments for corr matrices}, where correlations between matrix entries
were allowed, here all entries are independent.

Let $M$ be a sequence of non-Hermitian random matrices of order $N \times N$ satisfying some conditions. In particular, we assume

\begin{matrixcondition}[Matrix assumptions]\label{cond:matrixcond_1_NonHer}
Let $M= \frac{1}{\sqrt{N}}X =\frac{1}{\sqrt{N}}[x_{ij}]_{1\le i,j \le N}$ sequence of non-Hermitian real valued random matrices. Assume that the matrix $X$ satisfies the following conditions.
    \begin{enumerate}
        \item $\mathbb{E}[x_{ij}] = 0,\quad \mathbb{E}[x_{ij}^2] = 1$ for all $1 \le i,j \le N.$   
        \item All entries $\{x_{ij}\}_{1\le i,j\le N}$ are independent.
        \item The entries of $X$ have exploding moments in the sense that for all  $k \ge 2$,
$$
\frac{\mathbb{E}[x_{ij}^{k}]}{N^{k/2-\alpha}} \longrightarrow C_{k},\;\;\; i\neq j,
$$
and for all $k \ge 1$,
$$
\frac{\mathbb{E}[|x_{ii}|^{k}]}{N^{k/2-\alpha}} = O(1),
$$
for some $\alpha> 0$, where $(C_{k})_{k \ge 0}$ are finite constants.
  \end{enumerate}
\end{matrixcondition}

As in the previous Section \ref{section:Central Limit Theorem for Generalized Moments for corr matrices}, we expand
$$
\frac{1}{N}\Tr(M_N^k)
=
\sum_{\pi\in\mathcal P(k)} \tau_N^{nh}[T_\pi],
$$
where $T_\pi=(V, E)$ is the directed graph associated with the partition $\pi$ defined as similar to as in Section \ref{section:Central Limit Theorem for Generalized Moments for corr matrices}.

\begin{Thm}[Convergence of normalized trace]\label{thm:Convergence of normalized trace for non-Hermitian matrices with independent exploding entries}
Let $M$ be the matrix same as in Condition \ref{cond:matrixcond_1_NonHer}.
For any partition $\pi\in\mathcal P(k)$, let
$$
\tau_N^{nh}[T_\pi]
=
\frac{1}{N}
\sum_{\phi:V\to[N]\atop\mathrm{injective}}
\prod_{(u,v)\in E}
m_{\phi(u)\phi(v)},
$$
be defined for $M$ same as in \eqref{eq:tauN_def}.
Then, as $N\to\infty$, the following hold.
\begin{enumerate}
      \item For $\alpha>1$ and for every $T_\pi$ we have
    $$
    \mathbb{E}[\tau_N^{nh}[T_\pi]] \longrightarrow 0 .
    $$
    \item For $\alpha=1$ we have
    $$
    \mathbb{E}\big[\tau_N^{nh}[T_\pi]\big]
   \longrightarrow
    \tau^{nh}[T_\pi]
   :=
    \begin{cases}
   \displaystyle\prod_{k\ge2} C_{k}^{\,q_{k}}, & \text{if $T_\pi$ is a fat tree},\\[0.8em]
    0, & \text{otherwise,}
   \end{cases}
   $$
  where $q_{k}$ denote the number of unordered vertex pairs $\{i,j\}$ such that there are exactly $k$ edges between these vertices.
%    \item For $\alpha<1$ we have 
%    $$
% \mathbb{E}\big[\tau_N^0[T_\pi]\big] \sim O(N^{|V|-1-\alpha \mathfrak{p}}),
% $$
% where $|V|$ is the number of vertices of $T_\pi$, and $\mathfrak{p}$ is the number of edges of the tree obtained by forgetting the multiplicity of the edges of the graph $T_\pi$.
\item For $\alpha<1$ we have 
    \begin{align*}
\mathbb{E}\big[\tau_N^{nh}[T_\pi]\big] \sim
\begin{cases}
0, & \parbox[t]{0.6\textwidth}{\setstretch{0.9}\raggedright
$T_\pi$ contains any edge of multiplicity one (in any direction) or any vertex with a single loop (or both),}\\[0.6em]
O(N^{|V|-1-\alpha \mathfrak{p}}), & \text{otherwise,}
\end{cases}
\end{align*}
where $|V|$ is the number of vertices of $T_\pi$, and $\mathfrak{p}$ is the number of edges of the graph obtained by forgetting the multiplicity and orientation of the edges of the graph $T_\pi$.
\end{enumerate}
\end{Thm}
\begin{proof}
    All the notations used in this section are borrowed from the Section \ref{section:Central Limit Theorem for Generalized Moments for corr matrices}. For a fixed partition $\pi$, let $T_\pi=(V,E)$ be the associated directed graph. Then we have
% $$
% \tau_N^0[T_\pi]
% =\frac{1}{N}\sum_{\phi:V\rightarrow[N]}
% \mathbb{E}\bigg[\prod_{(u,v)\in E} a_{\phi(u)\phi(v)}\bigg].
% $$
\begin{align}\label{eq: tauN_NonHer}
\mathbb{E}\big[\tau_N^{nh}[T_\pi]\big]
% &=\frac{1}{N}\sum_{\substack{\phi:V\to [N]\\ \text{injective}}}
% \mathbb{E}\bigg[\prod_{(u,v)\in E} a_{\phi(u)\phi(v)}\bigg]\notag\\
    &=\frac{(N-1)!}{(N-|V|)!}\delta_N^0(T(m_N)),
\end{align}
where $\delta_N^0(T(m_N))=\mathbb{E}\bigg[\prod_{(u,v)\in E} m_{\phi(u)\phi(v)}\bigg].$
Since $\mathbb{E}[x_{ij}] = 0$, we have $\delta_N^0(T(m_N)) = 0$ whenever there exists an ordered pair of vertices $(u, v)$ such that they are connected by a directed edge of multiplicity one (in any direction), or there exists a vertex having a loop of multiplicity one. Therefore, to obtain a non zero contribution, we henceforth assume that the graphs $T_\pi$ do not contain any ordered pair $(u, v)$ connected by a directed edge of multiplicity one (in any direction), and no vertex has a loop of multiplicity one. We denote such graphs by $\hat{T}_{\pi}$. The reduced simple graph is denoted by $\hat{\mathcal{T}}_{\pi},$ which is obtained after ignoring the direction and multiplicity of the edges of $\hat{T}_{\pi}.$

By Condition \ref{cond:matrixcond_1_NonHer}, all the entries of the matrix are independent, all diagonal entries are identically distributed, and all off-diagonal entries are identically distributed. Let, for $k \ge 2$, let $p_k$ denote the number of vertices carrying exactly $k$ loops. For $k\ge 2$ let $q_{k}$ denote the number of unordered vertex pairs $\{i,j\}$ such that there are exactly $k$ edges between these vertices. Then 
\begin{align}
\label{eq: deltaN_NonHer}
    &\delta_N^0(T(m_N))
=
N^{-\alpha \mathfrak{p}}
\left(
\prod_{k\ge2} \left(N^{-\frac{k}{2}+\alpha }\mathbb{E}[|x_{11}|^k]\right)^{p_k}
\right)
\left(
\displaystyle\prod_{k\ge2}
\Big(
N^{-\frac{k}{2}+\alpha }
\mathbb{E}[x_{12}^{k}]
\Big)^{q_{k}}
\right),
\end{align}
where
$$
\mathfrak{p}
:=
\sum_{k\ge2} p_k
+
\displaystyle\sum_{k\ge2} q_{k}
$$
is the number of edges of $\hat{\mathcal{T}}_\pi.$ Now, by using \eqref{eq: deltaN_NonHer} in \eqref{eq: tauN_NonHer} we have
\begin{align}\label{eq: tauN_Final_NonHer}
\mathbb{E}\big[\tau_N^{nh}[\hat{T}_\pi]\big]
    %&\sim \frac{(N-1)!}{(N-|V|)!}\frac{1}{N^{\alpha \mathfrak{p}}}\left(
%\prod_{k\ge2} \left(N^{-\frac{k}{2}+\alpha }\mathbb{E}[|x_{11}|^k]\right)^{p_k}
%\right)
%\left(
%\displaystyle\prod_{k\ge2}
%\Big(
%N^{-\frac{k}{2}+\alpha }
%\mathbb{E}[x_{12}^k]
%\Big)^{q_{k}}
%\right)\notag\\
 &=N^{|V|-(\alpha \mathfrak{p}+1)}\left(
\prod_{k\ge2} \left(N^{-\frac{k}{2}+\alpha }\mathbb{E}[|x_{11}|^k]\right)^{p_k}
\right)
\left(
\displaystyle\prod_{k\ge2}
\Big(
N^{-\frac{k}{2}+\alpha }
\mathbb{E}[x_{12}^k]
\Big)^{q_{k}}
\right)
+o(1).
\end{align}
Since 
for all $k \ge 2$ 
% $$
%  \frac{\int z_1^kz_2^ld\mu_{X}(z)}{N^{(k+l)/2-1}} \longrightarrow C_{k,l}, \quad \frac{\int |z|^kd\nu_{X}(z)}{N^{k/2-1}} = O(1)
% $$ 
$$
N^{-\frac{k}{2}+\alpha }\mathbb{E}[x_{12}^k]
\longrightarrow C_{k}, 
$$
and for all $k\ge 2$ 
$$
N^{-\frac{k}{2}+\alpha }\mathbb{E}[|x_{11}|^k] = O(1),
$$
by Condition \ref{cond:matrixcond_1_NonHer}, the asymptotic behavior of \eqref{eq: tauN_Final_NonHer}
depends the power of $N$ that is on $N^{|V|-1-\alpha \mathfrak{p}}.$ Now, by following the same set of arguments as outlined in Theorem \ref{Thm:Convergence of normalized trace} we conclude the Theorem \ref{thm:Convergence of normalized trace for non-Hermitian matrices with independent exploding entries}. However, one should keep in mind that the underlying graph here is slightly different from the one used in \ref{Thm:Convergence of normalized trace}. Here, for any pair of vertices $\{u, v\},$ the edges $u\to v$ and $v\to u$ are independent, whereas those were correlated in the former case. This is reflected in the part (b) of the Theorems \ref{Thm:Convergence of normalized trace} and \ref{thm:Convergence of normalized trace for non-Hermitian matrices with independent exploding entries}. In the first case, it was thick tree, which is changed to fat tree in the current theorem.
\end{proof}
% Now, Theorem \ref{thm:CLT for normalized traces} holds for matrices which satisfy the Condition \ref{cond:matrixcond_1_NonHer} with $\tau_N^0[T_\pi]$ as defined in the Corollary \ref{thm:Convergence of normalized trace for non-Hermitian matrices with independent exploding entries}. This is explicitly stated in Theorem \ref{thm:CLT_iid_exploding}.
%%%%%%%%%%%%%%%
We now state CLT for the normalized traces.
\begin{Thm}[CLT for the normalized traces]\label{thm:CLT_iid_exploding}
Let $M$ be the matrix same as in Condition \ref{cond:matrixcond_1_NonHer} with $\alpha=1$. For any partition $\pi\in\bigcup_{k}\mathcal{P}(k)$, let $Z_N(T_\pi)$ be defined similar as in \eqref{eq:centered and normalized random variables Z_N} for matrix $M$.
Then the family
$
\big(Z_N(T_\pi)\big)_{\pi\in\bigcup_{k}\mathcal{P}(k)}
$
converges, as $N\to\infty$, to a centered Gaussian process
$\big(z(T_\pi)\big)_{\pi\in\bigcup_{k}\mathcal{P}(k)}$.
Moreover, for any two partitions $\pi_1,\pi_2$,
$$
\mathrm{Cov}\left(z(T_{\pi_1}),z(T_{\pi_2})\right)
=
\sum_{T\in\mathcal P^{\#}(T_{\pi_1},T_{\pi_2})}
\tau^{nh}[T],
$$
where $\mathcal{P}^{\#}(T_{\pi_1},T_{\pi_2})$ denotes the collection of graphs obtained by disjoint copies of the graphs $T_{\pi_1}$ and $T_{\pi_2}$ and gluing them with at least one common edge, and where $\tau^{nh}[T]$ is same as defined in the Theorem \ref{thm:Convergence of normalized trace for non-Hermitian matrices with independent exploding entries}.
\end{Thm}
\begin{proof}
    This follows by a similar argument to that in the proof of Theorem \ref{thm:CLT for normalized traces}.
\end{proof}
\begin{rem}
 Theorem \ref{thm:Convergence of normalized trace for non-Hermitian matrices with independent exploding entries} and Theorem \ref{thm:CLT_iid_exploding}, stated here, are not corollaries of Theorem \ref{Thm:Convergence of normalized trace} and Theorem \ref{thm:CLT for normalized traces}, respectively. This is because, in Theorem \ref{Thm:Convergence of normalized trace} and Theorem \ref{thm:CLT for normalized traces}, we have $C_{1,1} = \mathbb{E}[x_{ij}x_{ji}] \neq 0$ for $i \neq j$. But in Theorem \ref{thm:Convergence of normalized trace for non-Hermitian matrices with independent exploding entries}, $\mathbb{E}[x_{ij}x_{ji}]=\mathbb{E}[x_{ij}]\mathbb{E}[x_{ji}]=0$ for $i \neq j,$ because of the independence of the $(i,j)th$ and $(j,i)th$ entries.
\end{rem}

%%%%%%%%%%%%%%%%%%%%%%%%%%%%%%%%%%%%%%%%%%%%%%%%%%%%%%%%%%%%%%%%%%%%%%%%%%%
%%%%%%%%%%%%%%%%%%   Correlated block diagonal matrices  %%%%%%%%%%%%%%%%%%%%%
%%%%%%%%%%%%%%%%%%%%%%%%%%%%%%%%%%%%%%%%%%%%%%%%%%%%%%%%%%%%%%%%%%%%%%%%%%%

\section{Correlated block diagonal matrices with exploding moments}\label{sec:CLT for LES of block diagonal matrices with Exploding moments}
We consider a block diagonal matrix
\begin{align}\label{eqn: main definition of block matrix M}
    M=\left(\begin{array}{ll}M^{(1)} & 0 \\ 0 & M^{(2)}\end{array}\right)_{2N \times 2N},
\end{align}
where, for $t=1,2$, 
$$
M^{(t)}=[m_{ij}^{(t)}]_{1\le i,j\le N}=\frac{1}{\sqrt{N}}[x_{ij}^{(t)}]_{1\leq i,j \leq N }.
$$  We assume the following.
\begin{matrixcondition}\label{cond: Correlated Block matrix}
    \begin{enumerate}
    \item For $t=1,2$, let
    $$
    M^{(t)}=\frac{1}{\sqrt{N}}X^{(t)}=\frac{1}{\sqrt{N}}\,[x_{ij}^{(t)}]_{1\le i,j\le N}
    $$
    be non-Hermitian real-valued random matrices with independent entries satisfying Condition \ref{cond:matrixcond_1_NonHer}.

    \item For each $(i,j)$, where $i\neq j,$ the $(i,j)$th entry of $X^{(1)}$ is correlated with the corresponding $(i,j)$th entry of $X^{(2)}$ in the sense that, for all integers $k,l \ge 0$ with $k+l\geq 2$,
    $$
    \frac{\mathbb{E}\!\left[(x_{ij}^{(1)})^k (x_{ij}^{(2)})^l\right]}
    {N^{(k+l)/2-\alpha}}
    \longrightarrow C_{k,l},
    $$
     where $C_{k,0}=C_{0,l}=C_k$. For all $k \ge 1$ and $t=1,2$, the diagonal entries satisfy
    $$
    \frac{\mathbb{E}\left[|x_{ii}^{(t)}|^{k}\right]}
    {N^{k/2-\alpha}} = O(1),
    $$
    where $\alpha> 0$ and $(C_{k,l})_{k,l\ge0}$ are finite constants.
\end{enumerate} 
\end{matrixcondition}
The normalized trace of $M^k$ can be written as
\begin{align}\label{eq:block_trace_identity}
\frac{1}{N}\operatorname{Tr}(M^k)
&=
\frac{1}{N}\operatorname{Tr}((M^{(1)})^k)
+
\frac{1}{N}\operatorname{Tr}((M^{(2)})^k)
\notag\\
&= \frac{1}{N} \sum_{i_1,\dots,i_k = 1}^N 
(m_{i_1 i_2}^{(1)} m_{i_2 i_3}^{(1)} \cdots m_{i_k i_1}^{(1)}+m_{i_1 i_2}^{(2)} m_{i_2 i_3}^{(2)} \cdots m_{i_k i_1}^{(2)}).
\end{align}
Let $\mathcal{P}(k)$ denote the set of partitions of $\{1,2,\dots,k\}$. For each partition $\pi\in\mathcal P(k)$, define
$$
S_\pi
=
\left\{
(i_1,\dots,i_k)\in[N]^k:
 i_m=i_n \Longleftrightarrow  m\sim_\pi n
\right\}.
$$
%%%%%%%%%%%%%%%%%%%%%%%%%%%%%%%%%%%%%%%%%%%%%%%%%%%%%%%%%%%%%%%%%%%%%%%%%%%%

Then from \eqref{eq:block_trace_identity} we have
\begin{align}
\frac{1}{N}\operatorname{Tr}(M^k)
&=
\sum_{\pi\in\mathcal P(k)} \left(\tau_N^{(1)}[T_\pi]+
\tau_N^{(2)}[T_\pi]
\right),
\end{align}
where
$$
\tau_N^{(t)}[T_\pi]
=
\frac{1}{N}
\sum_{\substack{\phi:V\to[N]\\ \text{injective}}}
\prod_{(u,v)\in E}
m^{(t)}_{\phi(u)\phi(v)}.
$$

Here $T_\pi=(V, E)$ is the directed graph associated with $\pi$, constructed exactly as in Section \ref{section:Central Limit Theorem for Generalized Moments for corr matrices}.

\begin{defn}[Colored fat tree]
A connected directed graph $T=(V, E)$, whose each edge is colored by red or blue, is called colored fat tree if it satisfies the following properties.
Any directed edge $u\to v$ is accompanied by at least one more edge of any color in the same direction, and there is no edge in the reverse direction. It does not contain any vertex with a loop. Additionally, the underlying simple graph $\mathcal{T}=(V, \mathcal{E}),$ obtained by ignoring orientation and multiplicity satisfies $|V|=|\mathcal{E}|+1$.
\end{defn}

\begin{figure}[htbp]
\centering

% -------- First row (valid examples) --------
\begin{minipage}{0.45\textwidth}
\centering
\begin{tikzpicture}[->, >=stealth, thick, scale=0.9]

\node[circle, fill, inner sep=2pt] (v1) at (0,0) {};
\node[circle, fill, inner sep=2pt] (v2) at (1.8,1.2) {};
\node[circle, fill, inner sep=2pt] (v3) at (3.6,0) {};
\node[circle, fill, inner sep=2pt] (v4) at (1.8,-1.2) {};

% valid edges
\draw[blue, bend left=20] (v1) to (v2);
\draw[red, bend left=40] (v1) to (v2);

\draw[blue, bend left=20] (v2) to (v3);
\draw[red, bend left=40] (v2) to (v3);

\draw[blue, bend left=20] (v1) to (v4);
\draw[red, bend left=40] (v1) to (v4);

\end{tikzpicture}

\vspace{2mm}
\small (a) Colored fat tree
\end{minipage}
\hfill
\begin{minipage}{0.45\textwidth}
\centering
\begin{tikzpicture}[->, >=stealth, thick, scale=0.9]

\node[circle, fill, inner sep=2pt] (v1) at (0,0) {};
\node[circle, fill, inner sep=2pt] (v2) at (1.8,1.2) {};
\node[circle, fill, inner sep=2pt] (v3) at (3.6,0) {};
\node[circle, fill, inner sep=2pt] (v4) at (1.8,-1.2) {};
% \node[circle, fill, inner sep=2pt] (v5) at (5.4,1.2) {};

% valid edges
\draw[blue, bend left=20] (v1) to (v2);
\draw[red, bend left=40] (v1) to (v2);
% valid edges
\draw[blue, bend left=60] (v1) to (v2);
\draw[red, bend left=80] (v1) to (v2);

\draw[blue, bend left=20] (v2) to (v3);
\draw[red, bend left=40] (v2) to (v3);

\draw[blue, bend left=20] (v1) to (v4);
\draw[red, bend left=40] (v1) to (v4);

% \draw[blue, bend left=20] (v3) to (v5);
% \draw[red, bend left=40] (v3) to (v5);

\end{tikzpicture}

\vspace{2mm}
\small (b) Colored fat tree
\end{minipage}

\vspace{5mm}

% -------- Second row (violations) --------
\begin{minipage}{0.45\textwidth}
\centering
\begin{tikzpicture}[->, >=stealth, thick, scale=0.9]

\node[circle, fill, inner sep=2pt] (v1) at (0,0) {};
\node[circle, fill, inner sep=2pt] (v2) at (1.8,1.2) {};
\node[circle, fill, inner sep=2pt] (v3) at (3.6,0) {};

% violation: multiplicity one
\draw[red, ->] (v1) to (v2);

\draw[blue, bend left=20] (v2) to (v3);
\draw[blue, bend left=40] (v2) to (v3);

\end{tikzpicture}

\vspace{2mm}
\small (c) Violation: edge with multiplicity one
\end{minipage}
\hfill
\begin{minipage}{0.45\textwidth}
\centering
\begin{tikzpicture}[->, >=stealth, thick, scale=0.9]

\node[circle, fill, inner sep=2pt] (v1) at (0,0) {};
\node[circle, fill, inner sep=2pt] (v2) at (1.8,1.2) {};

\draw[blue, bend left=20] (v1) to (v2);
\draw[blue, bend left=40] (v1) to (v2);

% violation: loop
\draw[red] (v2) edge[loop above] ();

\end{tikzpicture}

\vspace{2mm}
\small (d) Violation: loop present
\end{minipage}

\vspace{5mm}

% -------- Third row (violations) --------
\begin{minipage}{0.45\textwidth}
\centering
\begin{tikzpicture}[->, >=stealth, thick, scale=0.9]

\node[circle, fill, inner sep=2pt] (v1) at (0,0) {};
\node[circle, fill, inner sep=2pt] (v2) at (1.8,1.2) {};
\node[circle, fill, inner sep=2pt] (v3) at (3.6,0) {};

% cycle
\draw[blue, bend left=20] (v1) to (v2);
\draw[blue, bend left=40] (v1) to (v2);

\draw[blue, bend left=20] (v2) to (v3);
\draw[blue, bend left=40] (v2) to (v3);

\draw[red, bend left=20] (v3) to (v1);
\draw[red, bend left=40] (v3) to (v1);

\end{tikzpicture}

\vspace{2mm}
\small (e) Violation: cycle present
\end{minipage}
\hfill
\begin{minipage}{0.45\textwidth}
\centering
\begin{tikzpicture}[->, >=stealth, thick, scale=0.9]

\node[circle, fill, inner sep=2pt] (v1) at (0,0) {};
\node[circle, fill, inner sep=2pt] (v2) at (1.8,1.2) {};
\node[circle, fill, inner sep=2pt] (v3) at (3.6,0) {};
\node[circle, fill, inner sep=2pt] (v4) at (1.8,-1.2) {};

\draw[blue, bend left=20] (v1) to (v2);
\draw[blue, bend left=40] (v1) to (v2);

\draw[blue, bend left=20] (v2) to (v3);
\draw[blue, bend left=40] (v2) to (v3);

\draw[blue, bend left=20] (v1) to (v4);
\draw[blue, bend left=40] (v1) to (v4);

% loop violation
\draw[red] (v2) edge[loop above] ();
\draw[red] (v2) edge[loop below] ();

\end{tikzpicture}

\vspace{2mm}
\small (f) Violation: loops present
\end{minipage}

\caption{
Examples illustrating colored fat trees (panels (a) and (b)) and violations of the defining conditions (panels (c)--(f)).
}
\label{fig:colored_fat_tree}
\end{figure}
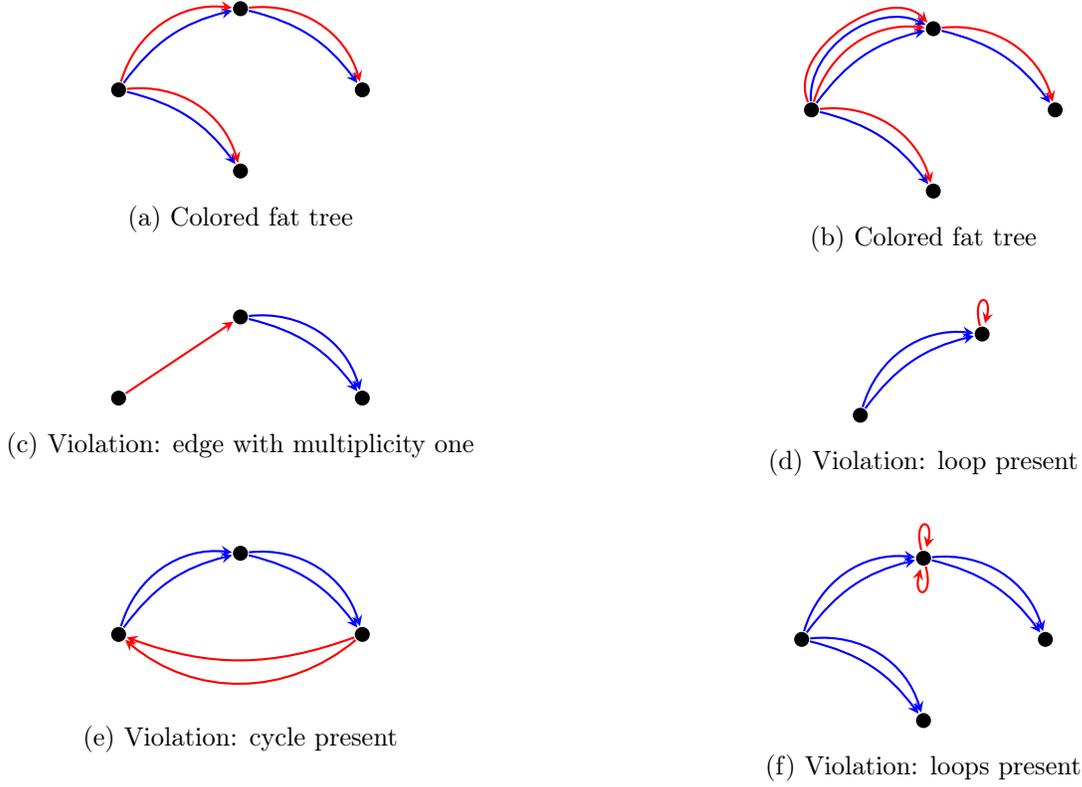

\subsection{CLT for the normalized traces}

To study fluctuations of the matrix $M$, we consider
\begin{align*}
    &\Big\{Z_N(k)\Big\}_{k \ge 1}\\
&= \left\{\frac{1}{\sqrt{N}}\operatorname{Tr}(M^k)
- \mathbb{E}\left[\frac{1}{\sqrt{N}}\operatorname{Tr}(M^k)\right]\right\}_{k \ge 1}\\
&=\left\{\frac{1}{\sqrt{N}}\Big(\operatorname{Tr}(M^{(1)})^k+\operatorname{Tr}(M^{(2)})^k\Big)
- \mathbb{E}\left[\frac{1}{\sqrt{N}}\Big(\operatorname{Tr}(M^{(1)})^k+\operatorname{Tr}(M^{(2)})^k\Big)\right]\right\}_{k \ge 1}.
\end{align*}
To prove the convergence of $\Big\{Z_N(k)\Big\}_{k \ge 1}$ to a Gaussian process, it is sufficient to prove the convergence of $\Big\{Z_N(T_{\pi})\Big\}_{\pi \in \mathcal{P}(k)},$ where
 \begin{align}\label{eq: Z_N T pi}
    &Z_N(T_{\pi})\\
&= \sqrt{N}\bigg(\big(\tau_N^{(1)}[T_{\pi}]+\tau_N^{(2)}[T_{\pi}]\big) - \big(\mathbb{E}[\tau_N^{(1)}[T_{\pi}]]+\mathbb{E}[\tau_N^{(2)}[T_{\pi}]]\big)\bigg)\notag\\
&= \sqrt{N}\Big(\tau_N^{(1)}[T_{\pi}] - \mathbb{E}[\tau_N^{(1)}[T_{\pi}]]\Big)+\sqrt{N}\Big(\tau_N^{(2)}[T_{\pi}]-\mathbb{E}[\tau_N^{(2)}[T_{\pi}]]\Big)\notag\\
:&=
Z_N^{(1)}(T_{\pi})
+
Z_N^{(2)}(T_{\pi}),\notag
\end{align}
to a Gaussian process, since
\begin{align*}
    \Big\{Z_N(k)\Big\}_{k \ge 1}= \Bigg\{\sum_{\pi \in \mathcal{P}(k)}Z_N\Big(T_{\pi}\Big)\Bigg\}_{k\ge 1}.
\end{align*}

We now state the CLT for the generalized moments of the matrix $M.$
\begin{Thm}[CLT for the normalized traces]\label{thm:CLT for normalized traces for correlated block diagonal matrix M}
Let $M$ be the block diagonal matrix defined in 
\eqref{eqn: main definition of block matrix M} satisfying 
Condition \ref{cond: Correlated Block matrix} with $\alpha=1$. 
For any partition $\pi\in \bigcup_{k\ge1}\mathcal P(k)$, let $Z_N(T_{\pi})$ be defined as in \eqref{eq: Z_N T pi}. Then the family of random variables
$
\big(Z_N(T_\pi)\big)_{\pi\in\bigcup_{k}\mathcal{P}(k)}
$
converges to a centered Gaussian process $\big(z(T_\pi)\big)_{\pi\in\bigcup_{k}\mathcal{P}(k)}$ as $N \rightarrow \infty$. 
Moreover, for any two partitions $\pi_1,\pi_2$, the limiting covariance is given by
\begin{align}
\mathrm{Cov}\left(z(T_{\pi_1}),z(T_{\pi_2})\right)
&=
\sum_{r_1,r_2 \in \{1,2\}}\sum_{T \in\mathcal{P}^{\#}(T_{\pi_1}^{(r_1)},T_{\pi_2}^{(r_2)})} \tau^b[T],
\end{align}
where the expression of $\tau^b[T]$ is given in \eqref{eq:covariance kerne for block matrix model}. Here $\mathcal{P}^{\#}(T_{\pi_1}^{(r_1)}, T_{\pi_2}^{(r_2)})$ denotes the collection of graphs obtained by disjoint copies of the graphs $T_{\pi_1}^{(r_1)}$ and $T_{\pi_2}^{(r_2)}$ and gluing them with at least one common edge. 
\end{Thm}

\begin{proof}
Let $\pi_1,\pi_2,\dots,\pi_l\in \bigcup_{k} \mathcal{P}(k).$
Let $T_1,T_2, \dots, T_r$ be directed graphs, each corresponding to a partition $\pi_j$, with $T_j = (V_j, E_j)$. These graphs are defined in the same way as in Section \ref{section:Central Limit Theorem for Generalized Moments for corr matrices}.
To prove convergence to a Gaussian process, it is sufficient to show that
$$
\mathbb{E}\Big[\prod_{j=1}^l Z_N(T_j)\Big]
$$
satisfies Wick’s formula. Recall from \eqref{eq: Z_N T pi}
$$
Z_N(T_j)
=
Z_N^{(1)}(T_j)+Z_N^{(2)}(T_j),
$$
where for $r_j=\{1,2\}$
$$
Z_N^{(r_j)}(T_j)
=
\sqrt{N}
\Big(
\tau_N^{(r_j)}[T_j]
-
\mathbb{E}[\tau_N^{(r_j)}[T_j]]
\Big),
$$
and for each $j$,
$$
\tau_N^{(r_j)}[T_j]
=
\frac{1}{N}
\sum_{\substack{\phi_j:V_j\to[N]\\ \mathrm{injective}}}
\prod_{(u,v)\in E_j}
m^{(r_j)}_{\phi_{j(u)}\phi_{j(v)}}.
$$
Let us define $\mathbf{r} = (r_1, r_2, \dots, r_l) \in \{1,2\}^l$.
We now consider
\begin{align}
&\mathbb{E}\Big[\prod_{j=1}^l Z_N(T_j)\Big]\\
&=
\mathbb{E}\Big[\prod_{j=1}^l
\Big(Z_N^{(1)}(T_j)+Z_N^{(2)}(T_j)\Big)\Big]
\nonumber\\
&=
\sum_{\mathbf{r}\in\{1,2\}^l}
\mathbb{E}
\Big[
\prod_{j=1}^l
Z_N^{(r_j)}(T_j)
\Big]\notag\\
&= \sum_{\mathbf{r}\in\{1,2\}^l}
(N)^{l/2}\mathbb{E}
\Bigg[
\prod_{j=1}^l
\Big(
\tau_N^{(r_j)}[T_j]
-\mathbb{E}[\tau_N^{(r_j)}[T_j]]
\Big)\Bigg]\notag\\
&= \sum_{\mathbf{r}\in\{1,2\}^l}
(N)^{l/2}\mathbb{E}
\Bigg[
\prod_{j=1}^l
\Bigg(
\frac{1}{N}\sum_{\substack{\phi_1,\phi_2,\dots,\phi_l \\ \phi_j:V_j \rightarrow [N] \text{ injective}}}
\prod_{(u,v)\in E_{j}}
m^{(r_j)}_{\phi_{j(u)}\phi_{j(v)}}\notag\\
&\;\;
-\mathbb{E}\bigg[\frac{1}{N}
\sum_{\substack{\phi_1,\phi_2,\dots,\phi_l\\ \phi_j:V_j \rightarrow [N] \text{ injective}}}
\prod_{(u,v)\in E_{j}}
m^{(r_j)}_{\phi_{j(u)}\phi_{j(v)}}\bigg]\Bigg)\Bigg]\notag\\
&=\frac{1}{(N)^{l/2}}
\sum_{\mathbf{r}\in\{1,2\}^l}
\sum_{\substack{\phi_1,\phi_2,\dots,\phi_l \\ \phi_j:V_j \rightarrow [N] \text{ injective}}}
\mathbb{E}
\Bigg[
\prod_{j=1}^l
\Bigg(
\prod_{(u,v)\in E_{j}}
m^{(r_j)}_{\phi_{j(u)}\phi_{j(v)}}\notag\\
&\;\;
-
\mathbb{E}\bigg[
\prod_{(u,v)\in E_{j}}
m^{(r_j)}_{\phi_{j(u)}\phi_{j(v)}}
\bigg]
\Bigg)
\Bigg].
\end{align}
Let $\mathcal{P}(V_1,V_2,\dots,V_l)$ be the set of partitions of disjoint union of $V_1\cup V_2\cdots\cup V_l$
such that each block contains at most one vertex from each $V_j$.
For $\sigma\in\mathcal{P}(V_1,V_2,\dots,V_l)$,
let
$$
S_\sigma =
\left\{
(\phi_1,\phi_2,\dots,\phi_l):
\phi_i(v)=\phi_j(v')
\iff v\sim_\sigma v'
\right\}.
$$
Then
\begin{align*}
    &\mathbb{E}\left[\prod_{j=1}^l Z_N(T_j)\right]\\
    &=(N)^{-l/2} \sum_{\mathbf{r}\in\{1,2\}^l}\sum_{\sigma\in\mathcal{P}(V_1,V_2,\dots,V_l)}
\sum_{(\phi_1,\phi_2,\dots,\phi_l)\in S_\sigma}
\omega_N(\phi_1,\phi_2,\dots,\phi_l;r_1,r_2,\dots,r_l),
\end{align*}
where
$$
\omega_N(\phi_1,\dots,\phi_r; \mathbf{r})
=
\mathbb{E}
\Bigg[
\prod_{j=1}^l
\bigg(
\prod_{(u,v)\in E_j}
m^{(r_j)}_{\phi_{j(u)}\phi_{j(v)}}-\mathbb{E}\bigg[\prod_{(u,v)\in E_j}
m^{(r_j)}_{\phi_{j(u)}\phi_{j(v)}}\bigg]
\bigg)
\Bigg].
$$
Since $\omega_N(\phi_1,\dots,\phi_l; r_1,r_2,\dots,r_l)$
depends only on $\sigma$ and the block labels
$(r_1,r_2,\dots,r_l)$. Hence, we can write
$$
\omega_N(\phi_1,\dots,\phi_r; \mathbf{r})=\omega_N(\sigma; \mathbf{r}).
$$
Moreover,
$$
\sum_{(\phi_1,\dots,\phi_l)\in S_\sigma} 1
=
\frac{N!}{(N-|\sigma|)!}=N^{|\sigma|}(1+O(N^{-1})).
$$

Therefore, 
\begin{align*}
\mathbb{E}\Big[\prod_{j=1}^l Z_N(T_j)\Big]
&=
\sum_{\sigma\in\mathcal P(V_1,\dots,V_l)}
\sum_{\mathbf{r}\in\{1,2\}^l}
N^{-l/2+|\sigma|}
\omega_N(\sigma; \mathbf{r})
(1+O(N^{-1}))\\
&=\sum_{\sigma\in\mathcal{P}^{\#}(V_1,\dots,V_l)}
\sum_{\mathbf{r}\in\{1,2\}^l}
N^{-l/2+|\sigma|}
\omega_N(\sigma; \mathbf{r})
(1+O(N^{-1})),
\end{align*}
where
\begin{align*}
&\mathcal{P}^{\#}(V_1,V_2,\dots,V_r)\\
&= \Big\{\sigma \in \mathcal{P}(V_1,V_2,\dots,V_r)\ \Big|\ 
\parbox{0.6\textwidth}{\setstretch{0.9}
each graph $T_i$ shares at least one edge with any of the other graphs $T_j$, $j\neq i$
}
\Big\}.
\end{align*}
We now consider 
\begin{align*}
\omega_N(\sigma; \mathbf{r})
&=
\sum_{B\subseteq\{1,2,\dots,l\}}
(-1)^{l-|B|}
\mathbb{E}
\bigg[
\prod_{j\in B}
\prod_{(u,v)\in E_j}
m^{(r_j)}_{\phi_{j(u)}\phi_{j(v)}}
\bigg]
\notag\\
&\qquad\times
\prod_{j\notin B}
\mathbb{E}
\bigg[
\prod_{(u,v)\in E_j}
m^{(r_j)}_{\phi_{j(u)}\phi_{j(v)}}
\bigg].
\end{align*}

For a fixed $B\subseteq\{1,2,\dots,r\}$, let ${T}_{B,\mathbf{r}}$ be the graph obtained by merging the vertices of the graphs $T_j$, $j\in B$, that belong to the same block of $\sigma$. We name the edges based on their weights: edges with weights from the matrix $M^{(1)}$ are referred as blue edges, and edges with weights from the matrix $M^{(2)}$ are referred as red edges. 

From now on, $\hat{T}$ denotes a graph that does not contain any pair of vertices $\{u,v\}$ connected by a directed edge of multiplicity one (in any direction), and does not contain any vertex with a single loop. The notation $\hat{\mathcal{T}}$ denotes the graph obtained from $\hat{T}$ by ignoring the multiplicities of its same colored edges. 

% \textcolor{blue}{Because of $\prod_{j\notin B}\mathbb{E}
% \bigg[
% \prod_{(u,v)\in E_j}
% m^{(r_j)}_{\phi_{j(u)}\phi_{j(v)}}
% \bigg]$ graph $\hat{T}$ must have at least two edge of same color. }

% \textcolor{orange}{Let $p_k$ denote the number of vertices of $\hat{T}_{B,\mathbf{r}}$ with $k$ loops attached, among which $p_k^{(1)}$ denotes the number of loops with weights from $M^{(1)}$ and $p_k^{(2)}$ denotes the number of loops with weights from $M^{(2)}$. Let $q_{k,l}$ denote the number of ordered vertex pairs $(u,v)$ with $u < v$ such that there are $k$ directed edges from $u$ to $v$ with weights from $M^{(1)}$ and $l$ directed edges from $u$ to $v$ with weights from $M^{(2)}$.} 
Let $p_k$ denote the number of vertices of $\hat{T}_{B,\mathbf{r}}$ that have $k$ loops attached. Among these, let $p_k^{(1)}$ denote the number of blue loops and $p_k^{(2)}$ denote the number of red loops. Let $q_{k,l}^{\rightarrow}$ denote the number of ordered vertex pairs $(u,v)$ with $u < v$ such that there are $k$ directed blue edges from $u$ to $v$ and $l$ directed red edges from $u$ to $v.$ Similarly, let $q_{k,l}^{\leftarrow}$ denote the number of ordered vertex pairs $(u,v)$ with $u < v$ such that there are $k$ directed blue edges from $v$ to $u$ and $l$ directed red edges from $v$ to $u$. By independence and the correlation assumptions on $M$, we have

\begin{align}
   &\mathbb{E}
\Bigg[
\prod_{j\in B}
\prod_{(u,v)\in E_{j}}
m^{(r_j)}_{\phi_{j(u)}\phi_{j(v)}}
\Bigg]\notag\\
&=\prod_{k\geq2}\big(\mathbb{E}[|m^{(1)}_{11}|^k]\big)^{p_k^{(1)}}\big(\mathbb{E}[|m^{(2)}_{11}|^k]\big)^{p_k^{(2)}}\displaystyle\prod_{\substack{k,l\ge0\\ k+l\ge2}}\big(\mathbb{E}[(m^{(1)}_{12})^{k}(m^{(2)}_{12})^{l}]\big)^{q_{k,l}^{\rightarrow}}\big(\mathbb{E}[(m^{(1)}_{21})^{k}(m^{(2)}_{21})^{l}]\big)^{q_{k,l}^{\leftarrow}}\notag\\
&= N^{-\mathfrak{p}_{B}}
\prod_{k\geq2}\bigg(\frac{\mathbb{E}[|x^{(1)}_{11}|^k]}{N^{k/2-1}}\bigg)^{p_k^{(1)}}\bigg(\frac{\mathbb{E}[|x^{(2)}_{11}|^k]}{N^{k/2-1}}\bigg)^{p_k^{(2)}}\displaystyle\prod_{\substack{k,l\ge0\\ k+l\ge2}}\bigg(\frac{\mathbb{E}[(m^{(1)}_{12})^{k}(m^{(2)}_{12})^{l}]}{N^{(k+l)/2-1}}\bigg)^{q_{k,l}^{\rightarrow}}\bigg(\frac{\mathbb{E}[(m^{(1)}_{21})^{k}(m^{(2)}_{21})^{l}]}{N^{(k+l)/2-1}}\bigg)^{q_{k,l}^{\leftarrow}}\notag\\
&=: N^{-\mathfrak{p}_{B}}\,\delta_N(B,\sigma),
\end{align}
where 
$$
\mathfrak{p}_{B}=\sum_{k\geq2}(p_k^{(1)}+p_k^{(2)})+\displaystyle\sum_{\substack{k,l\ge0\\ k+l\ge2}} \big(q_{k,l}^{\rightarrow}+q_{k,l}^{\leftarrow}\big).
$$ 
Here $\mathfrak{p}_{B}$ denotes the number of edges in the graph $\hat{\mathcal{T}}_{B,\mathbf{r}}$, which is obtained from $\hat{T}_{B,\mathbf{r}}$ by ignoring edge multiplicities while preserving directions. That means, for each ordered pair of vertices, we retain at most one directed edge in each direction. In particular, if multiple edges in the same direction are present, they are reduced to a single edge, and if edges in both directions are present, then both are retained, one in each direction.

Let ${T}_{B^{\complement},\mathbf{r}}$ be the graph obtained by merging the vertices of the graphs $T_j$, $j\notin B$,
that belong to the same block of $\sigma$. Let $p_{k}'$ denote the number of vertices of $\hat{T}_{B^{\complement},\sigma}$ with $k$ loops attached, out of which $(p_{k}')^{(1)}$ are the number of blue loops and $(p_{k}')^{(2)}$ denote the number of red loops.
Let $q_{k,l}'^{\rightarrow}$ denotes the number of ordered vertex pair $(u,v)$, with $u<v$
such that there are $k$ directed blue edges from $u$ to $v$ and $l$ directed red edges from $u$ to $v$. Similarly, let $q_{k,l}'^{\leftarrow}$ denotes the number of ordered vertex pair $(u,v)$, with $u<v$
such that there are $k$ directed blue edges from $v$ to $u$ and $l$ directed red edges from $v$ to $u$. By independence and the correlation assumptions on $M$, we have 
\begin{align*}
    &\prod_{j\notin B}
\mathbb{E}
\bigg[
\prod_{(u,v)\in E_j}
m^{(r_j)}_{\phi_{j(u)}\phi_{j(v)}}
\bigg]\\
&=\prod_{j\notin B} \Bigg\{  \prod_{k\geq2}\big(\mathbb{E}[(m^{(1)}_{11})^k]\big)^{(p_{k}')^{(1)}}\big(\mathbb{E}[(m^{(2)}_{11})^k]\big)^{(p_{k}')^{(2)}}\displaystyle\prod_{k,l\ge1}\Big(\mathbb{E}[(m^{(1)}_{12})^{k}]\mathbb{E}[(m^{(2)}_{12})^{l}]\Big)^{q_{k,l}'^{\rightarrow}}\\
&\;\;\;\; \times \Big(\mathbb{E}[(m^{(1)}_{21})^{k}]\mathbb{E}[(m^{(2)}_{21})^{l}]\Big)^{q_{k,l}'^{\leftarrow}}\Bigg\}\\
&=N^{-\mathfrak{p}_{B^\complement}}\prod_{j\notin B} \Bigg\{ \prod_{k\geq2}\bigg(\frac{\mathbb{E}[(x^{(1)}_{11})^k]}{N^{k/2-1}}\bigg)^{(p_k')^{(1)}}\bigg(\frac{\mathbb{E}[(x^{(2)}_{11})^k]}{N^{k/2-1}}\bigg)^{(p_k')^{(2)}}\notag\\
& \;\;\; \times\displaystyle\prod_{k,l\ge1}\bigg(\frac{\mathbb{E}[(m^{(1)}_{12})^{k}]\mathbb{E}[(m^{(2)}_{12})^{l}]}{N^{(k+l)/2-1}}\bigg)^{q_{k,l}'^{\rightarrow}}\bigg(\frac{\mathbb{E}[(m^{(1)}_{21})^{k}]\mathbb{E}[(m^{(2)}_{21})^{l}]}{N^{(k+l)/2-1}}\bigg)^{q_{k,l}'^{\leftarrow}} \Bigg\}\\
&=: N^{-\mathfrak{p}_{B^\complement}} \delta_{N}(B^\complement,\mathbf{r}).
\end{align*}
Here $\mathfrak{p}_{B^\complement}$ is the number of edges in the graph $\hat{\mathcal{T}}_{B^{\complement},\mathbf{r}}$.

Let $T_{\sigma,\mathbf{r}}$ be the directed graph obtained by merging all the vertices of
$T_1,T_2,\dots,T_l$ belonging to the same block of $\sigma$.
Let $c_\sigma$ be the number of connected components of $\hat{T}_{\sigma,\mathbf{r}}$ and
$\mathfrak{p}_{\sigma}$ the number of edges of $\hat{\mathcal{T}}_{\sigma,\mathbf{r}}$. Then

\begin{align}\label{eq: final exp product}
    &\mathbb{E}\left[\prod_{j=1}^l Z_N(T_j)\right]\\
    =&\sum_{\mathbf{r}\in\{1,2\}^l} \sum_{\sigma\in\mathcal{P}^\#(V_1,\dots,V_l)}
\sum_{B\subseteq\{1,2,\dots,l\}}  N^{-l/2+|\sigma|-\mathfrak{p}_{B}-\mathfrak{p}_{B^{\complement}}}(-1)^{l-|B|}
\delta_N(B,\mathbf{r})\delta_N(B^{\complement},\mathbf{r})\notag\\
& \;\;\; \times [1+O(N^{-1})]\notag\\
&=\sum_{\mathbf{r}\in\{1,2\}^l}\sum_{\sigma\in\mathcal{P}^\#(V_1,V_2,\dots,V_r)}N^{c_\sigma - r/2}N^{|\sigma|-c_\sigma-\mathfrak{p}_{\sigma}}
\sum_{B\subseteq\{1,2,\dots,r\}}
N^{\mathfrak{p}_{\sigma}-\mathfrak{p}_{B}-\mathfrak{p}_{B^\complement}}
(-1)^{r-|B|}\notag\\
&\;\;\; \times\delta_N(B,\mathbf{r})\delta_N(B^\complement,\mathbf{r})[1+O(N^{-1})].\notag
% &=\sum_{\mathbf{r}\in\{1,2\}^l}\sum_{\pi\in\mathcal{P}_2(r)}
% \prod_{\{i,j\}\in\pi}
% \sum_{T\in\mathcal{P}^{\#}(T_i,T_j)} \mathbb{E}[Z_N(T_i)Z_N(T_j)]
% +  o(1)\notag\\
% &=\sum_{\mathbf{r}\in\{1,2\}^l}\sum_{\pi\in\mathcal{P}_2(r)}
% \prod_{\{i,j\}\in\pi}
% \sum_{T\in\mathcal{P}^{\#}(T_i,T_j)} \mathbb{E}\big[ (Z_{N}^{(r_i)}(T_i)+ Z_{N}^{(r_i)}(T_i))(Z_{N}^{(r_j)}(T_j)+ Z_{N}^{(r_j)}(T_j)) \big]
% +  o(1).\notag
\end{align}
%%%%%%%%%%%%%%%%%%%%%%%%%%%%%%%%%%%%%%%%%%
Here the quantities $\delta_N(B,\mathbf{r})$ and $\delta_N(B^\complement,\mathbf{r})$ are bounded. Thus, the asymptotic behavior of \eqref{eq: final exp product}
depends on the exponent of $N$ that is on 
$$
(c_\sigma - l/2)+(\mathfrak{p}_{\sigma}-\mathfrak{p}_{B}-\mathfrak{p}_{B^\complement})+
(|\sigma|-c_\sigma-\mathfrak{p}_{\sigma}).
$$
We first consider $\sum_{B\subseteq\{1,2,\dots,l\}}
N^{\mathfrak{p}_{\sigma}-\mathfrak{p}_{B}-\mathfrak{p}_{B^\complement}}.$ Here $\mathfrak{p}_{\sigma}-\mathfrak{p}_{B}-\mathfrak{p}_{B^\complement} \le 0$ and equality holds iff $ B = \{1,2,\dots,l\} $ or $ B^\complement = \{1,2,\dots,l\} $. Without loss of generality, we assume $B = \{1,2,\dots,l\}.$ Then,
\begin{align}
    N^{\mathfrak{p}_\sigma-\mathfrak{p}_{B}-\mathfrak{p}_{B^\complement}}
= \mathbf{1}_{\{B=\{1,2,\dots,l\}\}} + O(N^{-1}).
\end{align}

We now consider the remaining factors $N^{c_\sigma - l/2}N^{|\sigma|-c_\sigma-\mathfrak{p}_\sigma}.$ Since $c_\sigma\leq l/2$ we have
$$
N^{c_\sigma-l/2} = O(N^{-\epsilon}) \text{ if }c_\sigma < l/2,
$$
where $\epsilon>0.$
Now, from \cite[Theorem 2.1.4]{west2001introduction}, $\mathfrak{p}_\sigma+c_\sigma-|\sigma|$ is the number of cycles of the graph $\hat{\mathcal{T}}_{\sigma,\mathbf{r}}$. Note that the graph $\hat{\mathcal{T}}_{\sigma,\mathbf{r}}$ may contain pair vertices $(u, v),$ where both the directed edges $u\to v$ and $v\to u$ are present. Such cases will be counted as a cycle in the factor $\mathfrak{p}_\sigma+c_\sigma-|\sigma|$. Since the number of cycles is always a non-negative number, we have $\mathfrak{p}_\sigma+c_\sigma-|\sigma|\geq0.$ Equality holds iff $\hat{\mathcal{T}}_{\sigma,\mathbf{r}}$ is a forest i.e., each component is a tree. These trees are free from the vertex pairs $(u, v)$, where both types of directed edges $u\to v$ and $v\to u$ are present. Thus
\begin{align*}
    N^{|\sigma|-c_\sigma-\mathfrak{p}_\sigma}= O(N^{-1}) \text{ if }|\sigma|-c_\sigma-\mathfrak{p}_\sigma < 0.
\end{align*}
Consequently, if either $ c_\sigma < l/2 $ or $\hat{\mathcal{T}}_{\sigma,\mathbf{r}}$ is not a forest, then 
$$
N^{c_\sigma - l/2}
N^{|\sigma| - c_\sigma - \mathfrak{p}_\sigma}
= O(N^{-1}).
$$
Thus, to obtain a non zero contribution, we need $c_\sigma=l/2$, and graph $\hat{\mathcal{T}}_{\sigma,\mathbf{r}}$ to be a forest. Hence 
\begin{align}
   & N^{|\sigma|-c_\sigma-\mathfrak{p}_\sigma}
= \mathbf{1}_{\{ \hat{\mathcal{T}}_\sigma\text{ is a forest}\}} + O(N^{-1}),
\end{align}
and
\begin{align}\label{eq: Connected components}
    N^{c_\sigma - l/2} =
\begin{cases}
1, & \text{if } c_\sigma=l/2\\[4pt]
O(N^{-1}), & \text{otherwise}.
\end{cases}
\end{align}
Each partition $\sigma\in \mathcal{P}^\#(V_1,V_2,\dots,V_l)$ induces another partition $\bar\sigma$ of $\{1,2,\dots,l\}$, defined by
$$
i \sim_{\bar\sigma} j
\quad \text{if and only if} \quad
T_i \text{ and } T_j \text{ belong to the same connected component of } \hat{T}_\sigma.
$$
Then we can write \eqref{eq: Connected components} as follows
\begin{align}
    N^{c_\sigma - l/2} =
\begin{cases}
1, & \text{if } \bar\sigma \in \mathcal{P}_2(l),\\[4pt]
O(N^{-1}), & \text{otherwise},
\end{cases}
\end{align}
where $\mathcal{P}_2(l)$ denotes the set of pair partition of $\{1,2,\dots,l\}.$

Therefore from \eqref{eq: final exp product} we have
\begin{align}
    &\mathbb{E}\left[\prod_{j=1}^l Z_N(T_j)\right]\\
    &=\sum_{\mathbf{r}\in\{1,2\}^l}
\sum_{\pi\in\mathcal{P}_2(l)}
\sum_{\substack{\sigma\in\mathcal{P}^{\#}(V_1,\dots,V_l)\\ \bar\sigma=\pi}}
\mathbf{1}_{\{\hat{\mathcal{T}}_{\sigma,\mathbf{r}} \text{ is a forest of } l/2 \text{ trees}\}}
\,\delta_N(\sigma,\mathbf{r})
+ O(N^{-1})\notag\\
%     &=
% \sum_{\pi\in\mathcal{P}_2(l)}
% \prod_{\{i,j\}\in\pi}
% \left(\sum_{r_i,r_j\in\{1,2\}}
% \sum_{\sigma\in\mathcal{P}^{\#}(V_i,V_j)}
% \mathbf{1}_{\{\hat{\mathcal{T}}_\sigma \text{ is a tree}\}}
% \,\delta_N(\sigma,r_i,r_j)
% \right)
% + o(1)\notag\\
&=
\sum_{\pi\in\mathcal{P}_2(l)}
\prod_{\{i,j\}\in\pi}
\left(\sum_{r_i,r_j\in\{1,2\}}
\sum_{\sigma\in\mathcal{P}^{\#}(V_i,V_j)}
\mathbf{1}_{\{\hat{\mathcal{T}}_{\sigma,r_i,r_j} \text{ is a tree}\}}
\,\delta_N(\sigma,r_i,r_j)
\right)
+  O(N^{-1})\notag\\
&=\sum_{\pi\in\mathcal{P}_2(l)}
\prod_{\{i,j\}\in\pi}
\sum_{T\in\mathcal{P}^{\#}(T_i,T_j)} \mathbb{E}[Z_N(T_i)Z_N(T_j)]
+  o(1).\notag
\end{align}

We now evaluate the covariance kernel. Let $l=2$
\begin{align*}
   & \mathbb{E}[Z_N(T_1)Z_N(T_2)]\notag\\
&=
\sum_{r_1,r_2 \in \{1,2\}}\sum_{\sigma\in\mathcal{P}^{\#}(V_1,V_2)}
\mathbf{1}_{\{\hat{\mathcal{T}}_\sigma \text{ is a tree}\}}
\,\delta(\sigma,r_1,r_2)
+ o(1)\notag\\
&=\sum_{r_1,r_2 \in \{1,2\}}\sum_{\hat{\mathcal{T}}_\sigma \in\mathcal{P}^{\#}(T_1^{(r_1)},T_2^{(r_2)})} \tau_0[\hat{\mathcal{T}}_\sigma ]+ o(1),
\end{align*}
where we assume that in $T_i^{(r_i)}$, the superscript $r_i$ indicates that the weights are taken from the matrix $M^{(r_i)}$, and 
\begin{align}\label{eq:covariance kerne for block matrix model}
    \delta(\sigma,\mathbf{r})&= \lim_{N\to\infty}\delta_N(\sigma,\mathbf{r})\notag\\
     &:=
    \begin{cases}
   \displaystyle\prod_{\substack{k,l\ge0\\ k+l\ge2}} C_{k,l}^{\,q_{k,l}}, &\text{if ${T}_{\sigma,\mathbf{r}}$ is a colored fat tree},\\[0.8em]
    0, & \text{otherwise.}
   \end{cases}
\end{align}
Here $q_{k,l}$ denotes the number of vertex pairs $\{u, v\}$ such that there are $k$ directed blue and $l$ directed red edges in the same direction; no edges in the reverse direction.

%%%%%%%%%%%%%%%%%%%%%%%%%%%%%%%%%%%%%%%%%%%
\end{proof}

\section[Structured ensembles]{Structured ensembles: Centrosymmetric and Circulant matrices}
In this section, we extend the analysis of correlated non-Hermitian random matrices with exploding moments to several structured random matrix ensembles. Structured matrices arise naturally in applications, including signal processing, time series analysis, wireless communications, and statistical physics. Unlike fully independent ensembles, these matrices exhibit deterministic algebraic symmetries that introduce additional dependencies among entries. It is therefore natural to ask whether the results established in the previous sections, particularly the central limit theorem for linear eigenvalue statistic, continue to hold under such structural constraints.

\subsection{Centrosymmetric Matrices} \cite{jana2024spectrum}
We consider a sequence of random centrosymmetric matrices
$$
M = \frac{1}{\sqrt{N}} X, 
\qquad 
X = [x_{ij}]_{1 \le i,j \le N}
$$
with real valued entries that satisfy the Condition \ref{Condtion: Centro_exploding}.
\begin{matrixcondition}[Matrix Assumptions]\label{Condtion: Centro_exploding}
Let $M = \frac{1}{\sqrt{N}} X,$ where $X= [x_{ij}]_{1 \le i,j \le N}$ satisfies the following conditions.
    \begin{enumerate}
        \item $\mathbb{E}[x_{ij}] = 0,\qquad \mathbb{E}[x_{ij}^2] = 1$ for all $1 \le i,j \le N.$   
        \item For all $1 \le i,j \le N$, \quad $x_{ij}= x_{N+1-i,N+1-j}$.
        \item  For every $(i,j)$,
$$
x_{ij} \text{ is independent of } 
\{ x_{kl} : (k,l) \notin \{(i,j), (N+1-i,\,N+1-j)\} \}.
$$
        \item The entries of $X$ have exploding moments in the sense that for all $k \ge 2$,
$$
\frac{\mathbb{E}[x_{ij}^k]}{N^{k/2 - \alpha}}
\longrightarrow C_k,
$$
 where $\alpha> 0$ and $(C_k)_{k \ge 2}$ are finite constants.
    \end{enumerate}
\end{matrixcondition}

% \textcolor{red}{In Chapter 2, we showed that for finite moment assumptions, the empirical spectral distribution converges to the circular law, and proved the central limit theorem for linear eigenvalue statistic. We now investigate the fluctuation behavior in the exploding moment regime described in Condition \ref{Condtion: Centro_exploding}.}

We now invoke the following theorem, which reduces the centrosymmetric matrix to a block-diagonal form.

\begin{Thm}\cite[Theorem 9]{weaver1985centrosymmetric}\label{thm:orthogonal_chapter5}
 \begin{enumerate}[label=(\alph*)]
 % Even sized case
     \item  If $M=\left(\begin{array}{ll}A & B \\ C & D\end{array}\right)$ is an $N \times N$ centrosymmetric matrix with $N=2s$ and $A, B, C$, and $D \text{ are } s\times$ s matrices, then $M$ is orthogonally similar to
     \begin{align*}
         Q^T M Q=\left(\begin{array}{cc}
    A+J C & 0 \\
    0 & A-J C
\end{array}\right),
     \end{align*}
     where 
     \begin{align*}
         Q  =  \sqrt{\frac{1}{2}}\left(\begin{array}{ll}\mathbb{I} & -J \\ J & \mathbb{I}\end{array}\right).
     \end{align*}

% Odd sized case
     \item If $M=\left(\begin{array}{lll}A & x & B \\ y & q & yJ \\ C & Jx & D\end{array}\right)$ is an $N \times N$ centrosymmetric matrix with $N=2s+1$, and $A, B, C, D$  are  $s \times s$ matrices, $x$ is $s \times 1$ matrix, $y$ is $1 \times s$ matrix, and $q$ is a scalar, then $M$ is orthogonally similar to
     \begin{align*}
         Q^{T} M Q=\left(\begin{array}{ccc}
			A+J C & \sqrt{2} x & 0 \\
			\sqrt{2} y & q & 0 \\
			0 & 0 & A-J C
		\end{array}\right),
     \end{align*}
		where
  \begin{align*}
      Q  =  \sqrt{\frac{1}{2}}\left(\begin{array}{lll} \mathbb{I} & 0 & -J \\ 0 & \sqrt{2} & 0\\ J & 0 & \mathbb{I}\end{array}\right).
  \end{align*}
 \end{enumerate}
 Here $\mathbb{I}$ is the identity matrix and $J$ is the counter-identity matrix. In both cases $Q^TQ=\mathbb{I}$.
\end{Thm}

For simplicity, we restrict to the even dimensional case $N=2N$ (the odd case can be handled similarly). Since orthogonal similarity preserves eigenvalues, the spectrum of $M$ coincides with the union of the spectra of the diagonal blocks. Thus we write
\begin{align*}
    Q^{T}MQ = \left(\begin{array}{cc}
        M^{(1)} & 0 \\
        0 & M^{(2)}
    \end{array}\right),
\end{align*}
The matrices $M^{(1)}$ and $M^{(2)}$ are non-Hermitian random matrices of order $ N$. Although their entries are uncorrelated at the variance scale, the two blocks are not independent due to the underlying centrosymmetric constraint. Since similarity transformations preserve eigenvalues, the spectrum of $M$ is the union of the spectra of the two blocks. The exploding moment Condition \ref{Condtion: Centro_exploding} on the entries of $M$ induces the following asymptotic structure on the blocks.

\begin{matrixcondition}\label{condtion: Block centrosymmtric matrix exploding moment}
    For $t=1,2$,
$$
M^{(t)}=\frac{1}{\sqrt N}[x_{ij}^{(t)}]_{1\le i,j\le N }.
$$
\begin{enumerate}

\item For each $t=1,2$ and for all $k\ge 2$,
$$
\frac{\mathbb{E}\left[(x_{ij}^{(t)})^k\right]}
{N^{k/2-\alpha}}
\longrightarrow \widetilde C_k^{(t)},
$$
as $N \rightarrow \infty,$ where $\alpha\ge0$, and 
$(\widetilde C_k^{(t)})_{k\ge2}$ are finite constants depending only on $C_k^{(t)},$ and are explicitly given by
$$
\widetilde C_k^{(1)}
=
\sum_{r=0}^k
\binom{k}{r}
C_r C_{k-r},
$$
and 
$$
\widetilde C_k^{(2)}
= (-1)^{k-r}
\sum_{r=0}^k
\binom{k}{r}
C_r C_{k-r},
$$
where $C_k$ is as in 
Condition \ref{Condtion: Centro_exploding}.
\item For all integers $k,l\ge0$ with $k+l\ge2$,
$$
\frac{
\mathbb{E}\left[(x_{ij}^{(1)})^k (x_{ij}^{(2)})^l\right]
}
{N^{(k+l)/2-\alpha}}
\longrightarrow \widetilde C_{k,l},
$$
as $N \rightarrow \infty,$ where the constants $\widetilde {C}_{k,l}$ are finite and are explicitly given by
$$
\widetilde{C}_{k,l}
=
\sum_{r=0}^k
\sum_{s=0}^l
\binom{k}{r}
\binom{l}{s}
(-1)^{\,l-s}
C_{r,s} C_{k-r,l-s},
$$
where $C_{k,l}$ are as in 
Condition \ref{Condtion: Centro_exploding}.

\item In particular, $(i,j)th$ entries of the two blocks $M^{(1)}$ and $M^{(2)}$ are uncorrelated at the variance scale, so
$$
\widetilde{C}_{1,1}=0.
$$

\item For distinct index pairs $(i,j)\neq(k,l)$,
$$
\big(x_{ij}^{(1)},x_{ij}^{(2)}\big)
\text{ is independent of }
\big(x_{kl}^{(1)},x_{kl}^{(2)}\big).
$$
\end{enumerate}
\end{matrixcondition}

\begin{cor}[CLT for Centrosymmetric Matrices with Exploding moments]
    Let $M = \frac{1}{\sqrt{N}} X$ be a $2N\times 2N$ random centrosymmetric matrix satisfying 
Conditions \ref{Condtion: Centro_exploding}. For any partition $\pi\in\bigcup_{k}\mathcal{P}(k)$, let $Z_N(T_\pi)$ be defined as in \eqref{eq: Z_N T pi}.
Then, as $N \to \infty$, the collection $\big(Z_N(T_\pi)\big)_{\pi\in\bigcup_{k}\mathcal{P}(k)}$ converges in
distribution to a centered Gaussian process $\big(z(T_\pi)\big)_{\pi\in\bigcup_{k}\mathcal{P}(k)}$. Moreover, for any two partitions $\pi_1,\pi_2$, the covariance of the limiting Gaussian process is given by
\begin{align*}
\mathrm{Cov}\left(z(T_{\pi_1}),z(T_{\pi_2})\right)
&=\sum_{r_1,r_2 \in \{1,2\}}\sum_{T \in\mathcal{P}^{\#}(T_{\pi_1}^{(r_1)},T_{\pi_2}^{(r_2)})} \tau^c[T ]+ o(1),
\end{align*}
where
\begin{align*}
\tau^c[T]
=
\begin{cases}
   \displaystyle\prod_{\substack{k,l\ge0\\ k+l\ge2}} \widetilde C_{k,l}^{\,q_{k,l}}, &\text{if $T$ is a colored fat tree},\\[0.8em]
    0, & \text{otherwise.}
   \end{cases}
\end{align*}
Here $q_{k,l}$ denotes the number of vertex pairs $\{u,v\}$ such that there are $k$ directed edges from $u$ to $v$ with weights from matrix $M^{(1)}$ (blue) and $l$ directed edges from $u$ to $v$  with weights from matrix $M^{(2)}$ (red) and there is no edge from $v$ to $u.$
\end{cor}

\begin{proof}
 We have assumed the centrosymmetric matrices with exploding moments, as described in the Condition \ref{Condtion: Centro_exploding}. Then, by using the Weavers Theorem \ref{thm:orthogonal_chapter5}, centrosymmetric matrices can be reduced to a block diagonal form consisting of two correlated blocks. And we assume that these two correlated blocks satisfy Condition \ref{condtion: Block centrosymmtric matrix exploding moment}. Notice that this Condition \ref{condtion: Block centrosymmtric matrix exploding moment} is equivalent to a correlated block diagonal model Condition \ref{cond: Correlated Block matrix} considered in the Section \ref{sec:CLT for LES of block diagonal matrices with Exploding moments}. Therefore, the central limit theorem for the generalized moments follows
directly.

% $$
% \mathrm{Cov}(Z(T_i),Z(T_j))
% =
% \frac12
% \sum_{r,s=1}^{2}
% \sum_{\sigma \in \mathcal P^\#(V_i^{(r)},V_j^{(s)})}
% \mathbf 1_{\{T_\sigma \text{ is a directed tree}\}}
% \,\delta(\sigma;r,s),
% $$
% where
% \begin{align}
% \delta(\sigma)
% =
% \prod_{k\ge1}
% \big(\widetilde C_k^{(1)}\big)^{p_k} \big(\widetilde C_k^{(2)}\big)^{p_k} 
% \prod_{\substack{k,l\ge0\\ k+l\ge2}}
% \big(D_{k,l}\big)^{q_{k,l}}.
% \end{align}
\end{proof}

\subsection{Circulant Matrices}

In this section, we establish the central limit theorem for linear eigenvalue 
statistic of random circulant matrices in the exploding moment regime.

Let
$$
C=[c_{ij}]_{0\le i,j\le N-1}
=\frac{1}{\sqrt{N}}\operatorname{Circ}(x_0,x_1,\dots,x_{N-1})
$$
be an $N\times N$ random circulant matrix generated by independent
real valued random variables $\{x_j\}_{j=0}^{N-1}$. The entries of $C$ satisfy
$$
c_{ij}=\frac{1}{\sqrt N}\,x_{(i-j)\bmod N}, 
\qquad 0\le i,j\le N-1 .
$$

Equivalently, the matrix can be written in the form
$$
C=\frac{1}{\sqrt N}
\begin{bmatrix}
x_0 & x_1 & x_2 & \dots & x_{N-1} \\
x_{N-1} & x_0 & x_1 & \dots & x_{N-2} \\
x_{N-2} & x_{N-1} & x_0 & \dots & x_{N-3} \\
\vdots & \vdots & \vdots & \ddots & \vdots \\
x_1 & x_2 & x_3 & \dots & x_0
\end{bmatrix}.
$$

Each row of $C$ is obtained from the previous row by a cyclic
right shift.

\begin{matrixcondition}[Circulant Ensemble with Exploding Moments]\label{condition:Circulant Ensemble with Exploding Moments}
Let
$$
C=\frac{1}{\sqrt{N}}\operatorname{Circ}(x_0,x_1,\dots,x_{N-1})
$$
be an $N\times N$ random circulant matrix generated by independent real valued random variables $\{x_i\}_{i=0}^{N-1}$. We assume the following.

\begin{enumerate}
    \item $\mathbb{E}[x_i]=0$, \quad $\mathbb{E}[x_i^2]=1$.
    \item $\{x_i\}_{0 \leq i \leq N-1}$ are independent.
    \item For all $k\ge2$, and $0 \leq i \leq N-1$
   $$
    \frac{\mathbb{E}[x_i^k]}{N^{k/2-\alpha}}
    \longrightarrow C_k,
    $$
    where $\alpha\ge0$ and $(C_k)_{k\ge2}$ are finite constants.
\end{enumerate}
\end{matrixcondition}
For circulant matrices, spectral analysis simplifies significantly. The eigenvalues of $C$ are given by
$$
\lambda_k=\frac{1}{\sqrt N}\sum_{j=0}^{N-1}x_j\omega_N^{jk},
\qquad
\omega_N=e^{2\pi i/N},
\qquad
k=0,1,\dots,N-1.
$$
We now consider 
\begin{align}\label{eq:Tr of circualnt matrix}
    \Tr(C^k)&=\sum_{m=0}^{N-1}\lambda_m^k\notag\\
    &=\sum_{m=0}^{N-1}\bigg( \frac{1}{\sqrt N}\sum_{j=0}^{N-1}x_j\omega_N^{jm}\bigg)^k\notag\\
    &=\sum_{m=0}^{N-1}\frac{1}{N^{k/2}}\sum_{j_1,\dots,j_k=0}^{N-1}x_{j_1}\cdots x_{j_k}\omega_N^{m(j_1+\cdots+j_k)}\notag\\
    &=\frac{1}{N^{k/2}}\sum_{j_1,\dots,j_k=0}^{N-1}x_{j_1}\cdots x_{j_k}\sum_{m=0}^{N-1}\omega_N^{m(j_1+\cdots+j_k)}\notag\\
    &=\frac{1}{N^{k/2-1}}\sum_{\substack{j_1,\dots,j_k=0\\ j_1+\cdots+j_k=0 (\text{mod }N)}}^{N-1}x_{j_1}\cdots x_{j_k}
\end{align}
where in the last equality we used the following identity
$$
\sum_{m=0}^{N-1}\omega_N^{mr}
=
\begin{cases}
N, & r= 0 \;(\text{mod }N),\\
0, & \text{otherwise}.
\end{cases}
$$
% \textcolor{blue}{
% \begin{proof}
% Let
% $$
% S_r=\sum_{m=0}^{N-1}\omega_N^{mr}.
% $$
% \textbf{Case 1: $r\equiv0\pmod N$.}
% If $r=kN$ for some integer $k$, then
% $$
% \omega_N^{mr}
% =
% e^{\frac{2\pi i}{N}m(kN)}
% =
% e^{2\pi i mk}
% =
% 1 .
% $$
% Therefore
% $$
% S_r=\sum_{m=0}^{N-1}1=N.
% $$
% \medskip
% \textbf{Case 2: $r\not\equiv0\pmod N$.}
% Let
% $$
% q=\omega_N^{r}.
% $$
% Then the sum becomes a geometric series
% $$
% S_r=\sum_{m=0}^{N-1} q^{m}.
% $$
% Since $r\not\equiv0\pmod N$, we have $q\neq1$. Using the formula for the
% sum of a geometric series,
% $$
% S_r=\frac{1-q^{N}}{1-q}.
% $$
% Now compute
% $$
% q^N=(\omega_N^{r})^N
% =\omega_N^{rN}
% =\left(e^{\frac{2\pi i}{N}}\right)^{rN}
% =e^{2\pi i r}
% =1.
% $$
% Therefore
% $$
% S_r=\frac{1-1}{1-q}=0.
% $$
% Combining the two cases, we obtain
% $$
% \sum_{m=0}^{N-1}\omega_N^{mr}
% =
% \begin{cases}
% N, & r\equiv0\pmod N,\\
% 0, & \text{otherwise}.
% \end{cases}
% $$
% \end{proof}
% }

Thus from \eqref{eq:Tr of circualnt matrix}, we have
\begin{align}\label{eq:exp of Tr of circualnt matrix}
    \mathbb{E}[\Tr(C^k)]
    =\frac{1}{N^{k/2-1}}\sum_{\substack{j_1,\dots,j_k=0\\ j_1+\cdots+j_k=0(\text{mod }N)}}^{N-1}\mathbb{E}[x_{j_1}x_{j_2}\cdots x_{j_k}]
\end{align}
Since $\mathbb{E}[x_i]=0$ and the variables are independent, therefore
$\mathbb{E}[x_{j_1}x_{j_2}\cdots x_{j_k}]$ is non zero only when each variable
appears at least twice. Let us consider
$i_1,i_2,\dots,i_r$ be the distinct indices with multiplicities
$m_1,m_2,\dots,m_r$ such that
$
m_1+m_2+\cdots+m_r=k, \; m_i\ge2.
$
Thus, from Condition \ref{condition:Circulant Ensemble with Exploding Moments}, we have
\begin{align*}
    \mathbb{E}[x_{j_1}x_{j_2}\cdots x_{j_k}]&=\prod_{a=1}^r \mathbb{E}[x_{i_a}^{m_a}]\\
    &=N^{k/2-r\alpha}\prod_{a=1}^r \bigg(\frac{\mathbb{E}[x_{i_a}^{m_a}]}{N^{{m_a}/2-\alpha}}\bigg)\\
    &\sim N^{k/2-r\alpha}\prod_{a=1}^r C_{m_a}.
\end{align*}
We now estimate the number of index tuples satisfying the constraint
$j_1+\cdots+j_k= 0 (\text{mod }N).$ Without the constraint, the number of ways to choose $r$ distinct indices is $O(N^r)$. Now, from \eqref{eq:exp of Tr of circualnt matrix} we have 
\begin{align*}
\mathbb{E}[\Tr(C^k)]
&=\frac{1}{N^{k/2-1}}\sum_{\substack{j_1,\dots,j_k=0\\ j_1+\cdots+j_k=0(\text{mod }N)}}^{N-1}\mathbb{E}[x_{j_1}x_{j_2}\cdots x_{j_k}]\\
&=\frac{1}{N^{k/2-1}}\sum_{\substack{i_1,\dots,i_r=0\\ i_1,\dots i_r \text{ are distinct} \\ m_a\ge2,\,\sum m_a=k\\m_1i_1+\cdots+m_ri_r=0(\text{mod }N)}}^{N-1}\frac{k!}{m_1!\cdots m_r!}\mathbb{E}[x_{i_i}^{m_1}\cdots x_{i_r}^{m_r}]\\
&=\frac{1}{N^{k/2-1}}\sum_{\substack{i_1,\dots,i_r=0\\ i_1,\dots i_r \text{ are distinct} \\ m_a\ge2,\,\sum m_a=k\\m_1i_1+\cdots+m_ri_r=0(\text{mod }N)}}^{N-1}\frac{k!}{m_1!\cdots m_r!}\mathbb{E}[x_{i_i}^{m_1}]\cdots \mathbb{E}[x_{i_r}^{m_r}]\\
% &=\frac{1}{N^{k/2-1}}\sum_{\substack{(m_1,\dots,m_r)\\ m_a\ge2,\,\sum m_a=k}}N(N-1)\dots (N-r+1)N^{-1}\frac{k!}{m_1!\cdots m_r!} N^{k/2-r\alpha}
% \prod_{a=1}^r C_{m_a} \\
% &\sim\frac{1}{N^{k/2-1}}\sum_{\substack{i_1,\dots,i_r=0\\ i_1,\dots i_r \text{ are distinct} \\ m_a\ge2,\,\sum m_a=k\\m_1i_1+\cdots+m_ri_r=0(\text{mod }N)}}^{N-1}\frac{k!}{m_1!\cdots m_r!}N^{k/2-r\alpha}\prod_{a=1}^r C_{m_a} \\
&\sim \frac{1}{N^{k/2-1}}
\sum_{\substack{(m_1,\dots,m_r)\\ m_a\ge2,\,\sum m_a=k}}
\frac{k!}{m_1!\cdots m_r!}
N^{r-1} N^{k/2-r\alpha}
\prod_{a=1}^r C_{m_a} \\
&= \sum_{\substack{(m_1,\dots,m_r)\\ m_a\ge2,\,\sum m_a=k}}
\frac{k!}{m_1!\cdots m_r!}
N^{r(1-\alpha)}
\prod_{a=1}^r C_{m_a}.
\end{align*}

Therefore, 
$$
\lim_{N\to\infty}\mathbb{E}[\Tr(C^k)]
=
\begin{cases}
\infty, & \alpha<1,\\[6pt]
\displaystyle
\sum_{\substack{(m_1,\dots,m_r)\\ m_a\ge2,\,\sum_{a=1}^{r}m_a=k\\ 1\leq r \leq \left\lfloor \frac{k}{2} \right\rfloor}}
\frac{k!}{m_1!\cdots m_r!}
\prod_{a=1}^r C_{m_a},
& \alpha=1,\\[12pt]
0, & \alpha>1.
\end{cases}
$$

We now move to the CLT for the normalized traces. 
Recall from \eqref{eq:Tr of circualnt matrix} that for every integer $k\ge1$,
\begin{equation}\label{eq:circ_trace_formula_clt}
\Tr(C_N^k)
=
\frac{1}{N^{k/2-1}}
\sum_{\substack{j_1,\dots,j_k=0\\
j_1+\cdots+j_k\equiv0\ (\mathrm{mod}\ N)}}^{N-1}
x_{j_1}x_{j_2}\cdots x_{j_k}.
\end{equation}

We now define the centered and normalized traces by
\begin{equation}\label{eq:ZN_circulant}
Z_N(k):=
\frac{1}{\sqrt{N}}\Tr(C_N^k)-\mathbb{E}\left[\frac{1}{\sqrt{N}}\Tr(C_N^k)\right].
\end{equation}
We now state the CLT for the normalized traces.
\begin{Thm}[CLT for the normalized traces]\label{thm:CLT_circulant_exploding}
    Let $C_N$ be the random circulant matrix satisfying
Condition \ref{condition:Circulant Ensemble with Exploding Moments}
with $\alpha=1$. Then for $k\ge1$,
$$
Z_N(k)
=
\frac{1}{\sqrt{N}}\Tr(C_N^k)-\mathbb{E}\left[\frac{1}{\sqrt{N}}\Tr(C_N^k)\right]
$$
converges in distribution to a centered Gaussian random variable $z(k)$. Moreover, the Covariance kernel is given as follows. 
\begin{align}
\mathrm{Cov}(z(k), z(l))
=
\begin{cases}
k!  , & \text{if } k=l, \\
0, & \text{if } k \neq l. 
\end{cases} +o(1),
\end{align}
\end{Thm}
\begin{proof}
    We prove the theorem using the Wick's formula. For fixed integers $k_1,\dots,k_r\ge1$, consider the joint moment, we have 
    $$
    \mathbb{E}\left[\prod_{i=1}^r Z_N(k_i)\right].
    $$
    From \eqref{eq:ZN_circulant} and \eqref{eq:circ_trace_formula_clt}, we have 
    \begin{align*}
        &\prod_{i=1}^r Z_N(k_i)\\
        &=\prod_{i=1}^{r}\Bigg(\frac{1}{\sqrt{N}}\Tr(C_N^{k_i})-\mathbb{E}\left[\frac{1}{\sqrt{N}}\Tr(C_N^{k_i})\right]\Bigg)\\
&=N^{-r/2}\prod_{i=1}^{r}\bigg(\Tr(C_N^{k_i})
-
\mathbb{E}\Big[\Tr(C_N^{k_i})\Big]
\bigg)\\
&=N^{-r/2}\prod_{i=1}^{r}\Bigg(\frac{1}{N^{k_i/2-1}}
\sum_{\substack{j^{(i)}_1,\dots,j^{(i)}_{k_i}=0\\
j^{(i)}_1+\cdots+j^{(i)}_{k_i}\equiv0\ (\mathrm{mod}\ N)}}^{N-1}
\Big(
x_{j^{(i)}_1}\cdots x_{j^{(i)}_{k_i}}
-
\mathbb{E}[x_{j^{(i)}_1}\cdots x_{j^{(i)}_{k_i}}]
\Big)
\Bigg)\\
&=N^{-r/2}\frac{1}{N^{\frac{1}{2}(\sum_{i=1}^{r}k_i)-r}}\prod_{i=1}^{r}\Bigg(
\sum_{\substack{j^{(i)}_1,\dots,j^{(i)}_{k_i}=0\\
j^{(i)}_1+\cdots+j^{(i)}_{k_i}\equiv0\ (\mathrm{mod}\ N)}}^{N-1}
\Big(
x_{j^{(i)}_1}\cdots x_{j^{(i)}_{k_i}}
-
\mathbb{E}[x_{j^{(i)}_1}\cdots x_{j^{(i)}_{k_i}}]
\Big)
\Bigg)\\
&=N^{r/2}\frac{1}{N^{\frac{1}{2}(\sum_{i=1}^{r}k_i)}}\prod_{i=1}^{r}\Bigg(
\sum_{\substack{j^{(i)}_1,\dots,j^{(i)}_{k_i}=0\\
j^{(i)}_1+\cdots+j^{(i)}_{k_i}\equiv0\ (\mathrm{mod}\ N)}}^{N-1}
\Big(
x_{j^{(i)}_1}\cdots x_{j^{(i)}_{k_i}}
-
\mathbb{E}[x_{j^{(i)}_1}\cdots x_{j^{(i)}_{k_i}}]
\Big)
\Bigg)\\
&=N^{r/2}\frac{1}{N^{\frac{1}{2}(\sum_{i=1}^{r}k_i)}}\sum_{\substack{
j^{(i)}_1,\dots,j^{(i)}_{k_i}=0\\
j^{(i)}_1+\cdots+j^{(i)}_{k_i}\equiv0\ (\mathrm{mod}\ N)\\
1\le i\le r
}}^{N-1}
\prod_{i=1}^{r}
\Big(
x_{j^{(i)}_1}\cdots x_{j^{(i)}_{k_i}}
-
\mathbb{E}[x_{j^{(i)}_1}\cdots x_{j^{(i)}_{k_i}}]
\Big)\\
&=:N^{r/2}\frac{1}{N^{\frac{1}{2}(\sum_{i=1}^{r}k_i)}}\sum_{\substack{
j^{(i)}_1,\dots,j^{(i)}_{k_i}=0\\
j^{(i)}_1+\cdots+j^{(i)}_{k_i}\equiv0\ (\mathrm{mod}\ N)\\
1\le i\le r
}}^{N-1}
\prod_{i=1}^{r}
A_i\big(j^{(i)}_1,\dots,j^{(i)}_{k_i}\big),
    \end{align*}
where 
$$
A_i\big(j^{(i)}_1,\dots,j^{(i)}_{k_i}\big)=x_{j^{(i)}_1}\cdots x_{j^{(i)}_{k_i}}
-
\mathbb{E}[x_{j^{(i)}_1}\cdots x_{j^{(i)}_{k_i}}]
$$
This implies
\begin{align}\label{eq:expectation of product}
\mathbb{E}\left[\prod_{i=1}^r Z_N(k_i)\right]
&=
N^{r/2}\frac{1}{N^{\frac{1}{2}(\sum_{i=1}^{r}k_i)}}
\sum_{\substack{
j^{(i)}_1,\dots,j^{(i)}_{k_i}=0\\
j^{(i)}_1+\cdots+j^{(i)}_{k_i}\equiv0\ (\mathrm{mod}\ N)\\
1\le i\le r
}}^{N-1}
\mathbb{E}\Bigg[
\prod_{i=1}^{r}
A_i\big(j^{(i)}_1,\dots,j^{(i)}_{k_i}\big)
\Bigg].
\end{align}

This product contains $r$ centered trace blocks i.e., $Z_{N}(k_{i})$. Since $\mathbb{E}[x_{j}]=0$, only those configurations, where every random variable appears at least twice, have nontrivial contributions. If each block is independent of the rest of the blocks, then after fixing a configuration, we obtain $r$ disjoint blocks. In that case, because of the self centering of the blocks themselves, the contribution is zero. In fact, even if a single block is independent of all the others, its expectation vanishes due to centering, and hence the entire contribution is zero. Therefore, to obtain a non-zero contribution, each block must be linked, that is sharing a common random variable, to at least one other block. 

It is evident that the sum in the right hand side of \eqref{eq:expectation of product} can be re-expressed as a sum of different types of linked-block configurations. In particular, the sum contains all possible terms from $\sum_{\pi\in \mathcal{P}_{2}(r)}\prod_{(i,j)\in \pi}\mathbb{E}[Z_{N}(k_{i})Z_{N}(k_{j})].$ Additionally, it also contains the terms where more than two blocks are linked. However, as discussed below, the latter case does not contribute anything asymptotically. We now analyze all the possible cases as follows.

\textbf{Case 1 (Exactly two blocks are linked).}
In this case, we assume that each block is linked to exactly one other block. Consequently, $r$ must be even. Consider the following example with four chains.
$$
(x_{i_1}x_{i_2}x_{i_3}x_{i_4})(x_{j_1}x_{j_2}x_{j_3}x_{j_4})(x_{t_1}x_{t_2}x_{t_3}x_{t_4})(x_{s_1}x_{s_2}x_{s_3}x_{s_4}).
$$

\textbf{Case 1(a) (Complete overlap of the chains).}
We link the first chain with the second chain as follows.
\begin{figure}[h]
    \begin{center}
\begin{tikzpicture}

   \def\x{0}
   \def\y{0}
  \filldraw (\x, \y) circle (0.06);
   \node[anchor = south] at (\x, \y) {$x_{i_1}$};
  \filldraw (\x+1, \y) circle (0.06);
   \node[anchor = south] at (\x+1, \y) {$x_{i_2}$};
  \filldraw (\x+2, \y) circle (0.06);
   \node[anchor = south] at (\x+2, \y) {$x_{i_3}$};
  \filldraw (\x+3, \y) circle (0.06);
   \node[anchor = south] at (\x+3, \y) {$x_{i_4}$};
  % \filldraw (\x+4, \y) circle (0.06);
  %  \node[anchor = south] at (\x+4, \y) {$x_{i_5j_1}$};
  \filldraw (\x+5, \y) circle (0.06);
   \node[anchor = south] at (\x+5, \y) {$x_{j_1}$};
  \filldraw (\x+6, \y) circle (0.06);
   \node[anchor = south] at (\x+6, \y) {$x_{j_2}$};
  \filldraw (\x+7, \y) circle (0.06);
   \node[anchor = south] at (\x+7, \y) {$x_{j_3}$};
  \filldraw (\x+8, \y) circle (0.06);
  \node[anchor = south] at (\x+8, \y) {$x_{j_4}$};
  % \filldraw (\x+9, \y) circle (0.06);
  % \node[anchor = south] at (\x+9, \y) {$x_{i_{10}i_1}$};
 \draw[thick][red] (\x, \y) -- (\x, \y-0.5) -- (5+\x, \y-0.5) -- (\x+5, \y); 
 \draw[thick][blue] (\x+1, \y) -- (\x+1, \y-0.7) -- (6+\x, \y-0.7) -- (\x+6, \y);
 \draw[thick][violet] (\x+2, \y) -- (\x+2, \y-0.9) -- (7+\x, \y-0.9) -- (\x+7, \y);
  \draw[thick][teal] (\x+3, \y) -- (\x+3, \y-1.1) -- (8+\x, \y-1.1) -- (\x+8, \y);
  % \draw[thick][orange] (\x+4, \y) -- (\x+4, \y-1.3) -- (9+\x, \y-1.3) -- (\x+9, \y);
  \node at (4.50, -2.35) {Complete overlap};
  \end{tikzpicture}
  \end{center} 
  %\caption{Sequential intra chain merging}
   % \label{fig:k=10}
\end{figure}

That means we have
$$
i_1=j_1,\qquad i_2=j_2, \qquad i_3=j_3, \qquad i_4=j_4.
$$

From the modular constraints,
$$i_1+i_2+i_3+i_4=0 (\mathrm{mod}\ N), \qquad j_1+j_2+j_3+j_4=0 (\mathrm{mod}\ N),$$  we obtain

$$i_1+i_2+i_3+i_4=0 (\mathrm{mod}\ N).$$

Thus, we have 3 free indices under this type of linking. If every element in the first chain is paired with an element in the second chain, then the number of free indices is
$$
\text{(length of the chains)}-1
$$degree of freedom.

\textbf{Case 1(b) (Partial chain linking).}
Suppose we link the chains only partially (i.e., not a complete linking), for example, as follows.

\begin{figure}[h]
    \begin{center}
\begin{tikzpicture}

   \def\x{0}
   \def\y{0}
  \filldraw (\x, \y) circle (0.06);
   \node[anchor = south] at (\x, \y) {$x_{i_1}$};
  \filldraw (\x+1, \y) circle (0.06);
   \node[anchor = south] at (\x+1, \y) {$x_{i_2}$};
  \filldraw (\x+2, \y) circle (0.06);
   \node[anchor = south] at (\x+2, \y) {$x_{i_3}$};
  \filldraw (\x+3, \y) circle (0.06);
   \node[anchor = south] at (\x+3, \y) {$x_{i_4}$};
  % \filldraw (\x+4, \y) circle (0.06);
  %  \node[anchor = south] at (\x+4, \y) {$x_{i_5j_1}$};
  \filldraw (\x+5, \y) circle (0.06);
   \node[anchor = south] at (\x+5, \y) {$x_{j_1}$};
  \filldraw (\x+6, \y) circle (0.06);
   \node[anchor = south] at (\x+6, \y) {$x_{j_2}$};
  \filldraw (\x+7, \y) circle (0.06);
   \node[anchor = south] at (\x+7, \y) {$x_{j_3}$};
  \filldraw (\x+8, \y) circle (0.06);
  \node[anchor = south] at (\x+8, \y) {$x_{j_4}$};
  % \filldraw (\x+9, \y) circle (0.06);
  % \node[anchor = south] at (\x+9, \y) {$x_{i_{10}i_1}$};
 \draw[thick][red] (\x, \y) -- (\x, \y-0.5) -- (5+\x, \y-0.5) -- (\x+5, \y); 
 \draw[thick][blue] (\x+1, \y) -- (\x+1, \y-0.7) -- (6+\x, \y-0.7) -- (\x+6, \y);
 \draw[thick][violet] (\x+2, \y) -- (\x+2, \y-0.9) -- (3+\x, \y-0.9) -- (\x+3, \y);
  \draw[thick][teal] (\x+7, \y) -- (\x+7, \y-0.9) -- (8+\x, \y-0.9) -- (\x+8, \y);
  % \draw[thick][orange] (\x+4, \y) -- (\x+4, \y-1.3) -- (9+\x, \y-1.3) -- (\x+9, \y);
  \node at (4.50, -2.35) { Partial chain linking};
  \end{tikzpicture}
  \end{center} 
  %\caption{Sequential intra chain merging}
   % \label{fig:k=10}
\end{figure}
That means we have
$$
i_1=j_1,\qquad i_2=j_2, \qquad i_3=i_4, \qquad j_3=j_4.
$$

From the modular constraints
$$i_1+i_2+i_3+i_4=0 (\mathrm{mod}\ N), \qquad j_1+j_2+j_3+j_4=0 (\mathrm{mod}\ N),$$ we have

$$i_1+i_2+2i_3=0 (\mathrm{mod}\ N), \qquad i_1+i_2+2j_3=0 (\mathrm{mod}\ N).$$

In this case, we obtain only 2 degrees of freedom, which is strictly less than in the complete linking case. This is because partial linking introduces additional dependencies, thereby reducing the number of free variables.

\textbf{Case 1(c) (More than two random variables are equal).}
In the previous case 1(a) and case 1(b), we considered the case where exactly two random variables are equal. We now consider the case where more than two random variables are equal. For this, consider the following type of linking.
\begin{figure}[h]
    \begin{center}
\begin{tikzpicture}

   \def\x{0}
   \def\y{0}
  \filldraw (\x, \y) circle (0.06);
   \node[anchor = south] at (\x, \y) {$x_{i_1}$};
  \filldraw (\x+1, \y) circle (0.06);
   \node[anchor = south] at (\x+1, \y) {$x_{i_2}$};
  \filldraw (\x+2, \y) circle (0.06);
   \node[anchor = south] at (\x+2, \y) {$x_{i_3}$};
  \filldraw (\x+3, \y) circle (0.06);
   \node[anchor = south] at (\x+3, \y) {$x_{i_4}$};
  % \filldraw (\x+4, \y) circle (0.06);
  %  \node[anchor = south] at (\x+4, \y) {$x_{i_5j_1}$};
  \filldraw (\x+5, \y) circle (0.06);
   \node[anchor = south] at (\x+5, \y) {$x_{j_1}$};
  \filldraw (\x+6, \y) circle (0.06);
   \node[anchor = south] at (\x+6, \y) {$x_{j_2}$};
  \filldraw (\x+7, \y) circle (0.06);
   \node[anchor = south] at (\x+7, \y) {$x_{j_3}$};
  \filldraw (\x+8, \y) circle (0.06);
  \node[anchor = south] at (\x+8, \y) {$x_{j_4}$};
  % \filldraw (\x+9, \y) circle (0.06);
  % \node[anchor = south] at (\x+9, \y) {$x_{i_{10}i_1}$};
 \draw[thick][red] (\x, \y) -- (\x, \y-0.5) -- (5+\x, \y-0.5) -- (\x+5, \y); 
 \draw[thick][blue] (\x+1, \y) -- (\x+1, \y-0.7) -- (6+\x, \y-0.7) -- (\x+6, \y);
 \draw[thick][violet] (\x+2, \y) -- (\x+2, \y-0.9) -- (7+\x, \y-0.9) -- (\x+7, \y);
  \draw[thick][violet] (\x+3, \y) -- (\x+3, \y-0.9) -- (8+\x, \y-0.9) -- (\x+8, \y);
  % \draw[thick][orange] (\x+4, \y) -- (\x+4, \y-1.3) -- (9+\x, \y-1.3) -- (\x+9, \y);
  \node at (4.50, -2.35) {Not competently linked};
  \end{tikzpicture}
  \end{center} 
  %\caption{Sequential intra chain merging}
   % \label{fig:k=10}
\end{figure}
That means
$$
i_1=j_1,i_2=j_2, i_3=i_4=j_3=j_4
$$
In this case, we again obtain only $2$ degrees of freedom. The reason is that equating more variables creates additional dependencies among the indicesand consequently decreases the degrees of freedom. However, it should be noted that if a random variable appears with multiplicity more than two, then due to the nature of exploding moment, as mentioned in Condition \ref{condition:Circulant Ensemble with Exploding Moments}, one gets an exploding contribution of $N^{k/2-1}C_{k}.$ Nevertheless, this contribution can not compensate for the lost degree of freedom. Because every additional power of a random variable amounts to loss of one additional degree of freedom, which is a multiplication by $N^{-1}$. By contrast, from the exploding moment nature of the random variable, only $N^{1/2}$ is gained per additional power. Therefore overall, per additional constraint, it is a loss of $1/2$ degrees of freedom.

From this, we conclude that if two chains are linked, they must be linked completely to get the leading order contribution. Any other type of partial linking yields a strictly smaller contribution.

Therefore, in the following chains
$$
(x_{i_1}x_{i_2}x_{i_3}x_{i_4})(x_{j_1}x_{j_2}x_{j_3}x_{j_4})(x_{t_1}x_{t_2}x_{t_3}x_{t_4})(x_{s_1}x_{s_2}x_{s_3}x_{s_4}),
$$
if we overlap the first chain completely with any of the remaining chains, and the other two chains are overlapped completely, then each pair of chains contributes $3$ degrees of freedom, and hence the total number of degrees of freedom is $6.$

\textbf{Case 2 (Linking more than two chains).}

% Consider the following example with four chains.
% $$
% (x_{i_1}x_{i_2}x_{i_3}x_{i_4})(x_{j_1}x_{j_2}x_{j_3}x_{j_4})(x_{t_1}x_{t_2}x_{t_3}x_{t_4})(x_{s_1}x_{s_2}x_{s_3}x_{s_4}).
% $$

We now link more than two chains. Let us link the first three chains as shown above, while the fourth chain is linked within itself.

\begin{figure}[h]
    \begin{center}
\begin{tikzpicture}

   \def\x{0}
   \def\y{0}
  \filldraw (\x, \y) circle (0.06);
   \node[anchor = south] at (\x, \y) {$x_{i_1}$};
  \filldraw (\x+1, \y) circle (0.06);
   \node[anchor = south] at (\x+1, \y) {$x_{i_2}$};
  \filldraw (\x+2, \y) circle (0.06);
   \node[anchor = south] at (\x+2, \y) {$x_{i_3}$};
  \filldraw (\x+3, \y) circle (0.06);
   \node[anchor = south] at (\x+3, \y) {$x_{i_4}$};
  % \filldraw (\x+4, \y) circle (0.06);
  %  \node[anchor = south] at (\x+4, \y) {$x_{i_5j_1}$};
  \filldraw (\x+5, \y) circle (0.06);
   \node[anchor = south] at (\x+5, \y) {$x_{j_1}$};
  \filldraw (\x+6, \y) circle (0.06);
   \node[anchor = south] at (\x+6, \y) {$x_{j_2}$};
  \filldraw (\x+7, \y) circle (0.06);
   \node[anchor = south] at (\x+7, \y) {$x_{j_3}$};
  \filldraw (\x+8, \y) circle (0.06);
  \node[anchor = south] at (\x+8, \y) {$x_{j_4}$};
  \filldraw (\x+10, \y) circle (0.06);
  \node[anchor = south] at (\x+10, \y) {$x_{t_1}$};
  \filldraw (\x+11, \y) circle (0.06);
  \node[anchor = south] at (\x+11, \y) {$x_{t_2}$};
  \filldraw (\x+12, \y) circle (0.06);
  \node[anchor = south] at (\x+12, \y) {$x_{t_3}$};
  \filldraw (\x+13, \y) circle (0.06);
  \node[anchor = south] at (\x+13, \y) {$x_{t_4}$};
  % \filldraw (\x+9, \y) circle (0.06);
  % \node[anchor = south] at (\x+9, \y) {$x_{i_{10}i_1}$};
 \draw[thick][red] (\x, \y) -- (\x, \y-0.5) -- (5+\x, \y-0.5) -- (\x+5, \y); 
 \draw[thick][blue] (\x+1, \y) -- (\x+1, \y-0.7) -- (6+\x, \y-0.7) -- (\x+6, \y);
 \draw[thick][violet] (\x+2, \y) -- (\x+2, \y-0.9) -- (7+\x, \y-0.9) -- (\x+7, \y);
  \draw[thick][teal] (\x+3, \y) -- (\x+3, \y-1.1) -- (10+\x, \y-1.1) -- (\x+10, \y);
   \draw[thick][orange] (\x+8, \y) -- (\x+8, \y-0.5) -- (11+\x, \y-0.5) -- (\x+11, \y);
   \draw[thick][black] (\x+12, \y) -- (\x+12, \y-0.7) -- (13+\x, \y-0.7) -- (\x+13, \y);
  \node at (7.50, -2.35) {Three chain link};
  \end{tikzpicture}
  \end{center} 
  %\caption{Sequential intra chain merging}
   % \label{fig:k=10}
\end{figure}

\begin{figure}[H]
    \begin{center}
\begin{tikzpicture}

   \def\x{0}
   \def\y{0}
  \filldraw (\x, \y) circle (0.06);
   \node[anchor = south] at (\x, \y) {$x_{s_1}$};
  \filldraw (\x+1, \y) circle (0.06);
   \node[anchor = south] at (\x+1, \y) {$x_{s_2}$};
  \filldraw (\x+2, \y) circle (0.06);
   \node[anchor = south] at (\x+2, \y) {$x_{s_3}$};
  \filldraw (\x+3, \y) circle (0.06);
   \node[anchor = south] at (\x+3, \y) {$x_{s_4}$};
 \draw[thick][red] (\x, \y) -- (\x, \y-0.5) -- (1+\x, \y-0.5) -- (\x+1, \y); 
 %\draw[thick][blue] (\x+1, \y) -- (\x+1, \y-0.7) -- (6+\x, \y-0.7) -- (\x+6, \y);
 \draw[thick][violet] (\x+2, \y) -- (\x+2, \y-0.5) -- (3+\x, \y-0.5) -- (\x+3, \y);
  \node at (1.60, -1.35) { Forth chain};
  \end{tikzpicture}
  \end{center} 
  %\caption{Sequential intra chain merging}
   % \label{fig:k=10}
\end{figure}

From this type of linking, the reduced modular constraints are
$$
i_1+i_2+i_3+i_4=0\;(\mathrm{mod}\ N),\qquad i_1+i_2+i_3+j_4=0\;(\mathrm{mod}\ N), 
$$
$$
i_4+j_4+2t_3=0\;(\mathrm{mod}\ N), \qquad 2s_1+2s_2=0\;(\mathrm{mod}\ N).
$$

Therefore, this configuration yields only $4$ degrees of freedom, which is strictly less than in Case 1, where we had complete pairwise linking of two chains. The reason is that linking more than two chains introduces additional dependencies among the indices, thereby reducing the number of free variables. Hence, such configurations contribute strictly less than the leading order.

Therefore, from Case 1 and Case 2, we conclude that if more than two chains are connected through the pairing. For instance, consider a configuration in which three chains are linked together. In such a case, the indices from these chains become coupled through the pairing, which creates additional dependencies among them. As a result, several indices are forced to be equal, and the corresponding modular constraints are no longer independent. This reduces the number of free variables compared to the case where chains are paired disjointly. In such cases, the total number of free indices is strictly less than in the case where exactly two chains are linked completely. Consequently, the contribution from such configurations is of lower order. Therefore, these configurations contribute strictly less compared to the pairing case (i.e., when chains are connected in disjoint pairs).

\textbf{Exact degree of freedom in complete overlapping}

We now determine the exact degree of freedom in the case where exactly two chains are completely linked. Consider $r$ chains, each of the same length, given by

$$
(x_{i_1^{(1)}}x_{i_2^{(1)}}x_{i_3^{(1)}}\dots x_{i_{k_1}^{(1)}})(x_{i_1^{(2)}}x_{i_2^{(2)}}x_{i_3^{(2)}}\dots x_{i_{k_1}^{(2)}})\dots (x_{i_1^{(r)}}x_{i_2^{(r)}}x_{i_3^{(r)}}\dots x_{i_{k_1}^{(r)}}).
$$

Suppose that the chains are paired and each pair is completely linked. Then the reduced modular constraints become

\begin{align*}
&i_1^{(1)}+i_2^{(1)}+\cdots+i_{k_1}^{(1)} \equiv 0 \ (\mathrm{mod}\ N),\\
&i_1^{(2)}+i_2^{(2)}+\cdots+i_{k_1}^{(2)} \equiv 0 \ (\mathrm{mod}\ N),\\
&\vdots \\
&i_1^{(r/2)}+i_2^{(r/2)}+\cdots+i_{k_1}^{(r/2)} \equiv 0 \ (\mathrm{mod}\ N).
\end{align*}
Thus, each pair contributes $(k_1 - 1)$ degrees of freedom. Therefore, the total number of degrees of freedom is 
$$
(k_1-1)+(k_2-1)+\dots (k_{r/2}-1).
$$
From \eqref{eq:expectation of product}, the corresponding exponent of $N$ is 
$$
N^{r/2}N^{-\frac{1}{2}(\sum_{i=1}^{r}k_i)}N^{(k_1-1)+(k_2-1)+\dots (k_{r/2}-1)}
$$
which is of order $O(1),$ since all chains have the same length $k_1.$ Therefore, when exactly two chains are completely linked (i.e., chains are paired), the contribution is of constant order.

On the other hand, as shown in Case 1 and Case 2, any configuration in which more than two chains are linked leads to fewer degrees of freedom. In such cases, the total number of free indices is strictly less than 
$$\frac{1}{2}\sum_{i=1}^r k_i-\frac{r}{2}.$$ Consequently, all configurations other than pairwise linking contribute at most
$$
O(N^{-\epsilon}), \epsilon>0
$$

Thus
\begin{align*}
\mathbb{E}\left[\prod_{i=1}^{2m} Z_N(k_i)\right]
=
\sum_{\pi\in\mathcal{P}_2(2m)}
\prod_{(i,j)\in\pi}
\mathbb{E}[Z_N(k_i)Z_N(k_j)]+o(1).
\end{align*}

%%%%%%%%%%%%%%%%%%%%%%%%%%%%%%%%%%%%%%%%%%%%%%%%%%%%%%%%%%%%%%%%%%%%%%%%%%%%%%%%%%%%%%%%%%%%%%%      Calculation of the covariance kernel              %%%%%%%%%%%%%%%%%%%%%%%%%%%%%%%%%%%%%%%%%%%%%%%%%%%%%%%%%%%%%%%%%%%%%%%%%%%%%%%%%%%%%%%%%%

\textbf{Calculation of the covariance kernel}

From \eqref{eq:expectation of product}, taking $r=2$, we obtain
\begin{align}
&\mathbb{E}[Z_N(k)Z_N(l)]\\
&=
N\frac{1}{N^{(k+l)/2}}
\sum_{\substack{
j_1,\dots,j_k=0\\
j_1+\cdots+j_k=0\ (\mathrm{mod}\ N)
}}^{N-1}
\sum_{\substack{
i_1,\dots,i_l=0\\
i_1+\cdots+i_l=0\ (\mathrm{mod}\ N)
}}^{N-1}
\mathbb{E}\Big[
A_1(j_1,\dots,j_k)\,A_2(i_1,\dots,i_l)
\Big].
\end{align}
Now the expectation is nonzero only when each index among
$$
\{j_1,\dots,j_k,i_1,\dots,i_l\}
$$
appears exactly twice.

\textit{Case 1: $k \neq l$.} In this case, it is not possible to pair all indices between the two blocks so that each index appears exactly twice. Hence,
$$
\mathbb{E}[Z_{N}(k)Z_{N}(l)] = 0.
$$
\textit{Case 2: $k = l$.} In this case, the only contributing configurations are those in which each index from the first block is paired with exactly one index from the second block. Such configurations correspond to bijections between $\{1,\dots,k\}$ and $\{1,\dots,k\}$, and their number is $k!$. Therefore, we have
% \begin{align}
% &\mathbb{E}[Z_N(k)Z_N(l)]\\
% &=
% \frac{1}{N^{k-1}}
% \sum_{\substack{
% i_1,\dots,i_l=0\\
% i_1+\cdots+i_l=0\ (\mathrm{mod}\ N)
% }}^{N-1}
% \mathbb{E}\Big[
% A_1(i_1,\dots,i_k)A_2(i_1,\dots,i_k)
% \Big].
% \end{align}

\begin{align}
\mathbb{E}[Z_{N}(k)Z_{N}(l)]
=
\begin{cases}
k! \, (\mathbb{E}[x_0^2])^k + o(1), & \text{if } k=l, \\
0, & \text{if } k \neq l.
\end{cases}
\end{align}
\end{proof}

We believe that, by using techniques similar to those discussed here for circulant matrices, we can also study the spectra of reverse circulant matrices, symmetric circulant matrices, Toeplitz matrices, and Hankel matrices.

\subsection*{Acknowledgment} Indrajit Jana's research is partially supported by INSPIRE Fellowship\\DST/INSPIRE/04/2019/000015, Dept. of Science and Technology, Govt. of India.\\

Sunita Rani's research is fully supported by the University Grant Commission (UGC), New Delhi.

\bibliographystyle{abbrv} % We choose the "abbrv" reference style
\bibliography{MasterBib} % Entries are in the MasterBib.bib file

\begin{thebibliography}{10}

\bibitem{adhikari2017fluctuations}
K.~Adhikari and K.~Saha.
\newblock Fluctuations of eigenvalues of patterned random matrices.
\newblock {\em Journal of Mathematical Physics}, 58(6):063301, 2017.

\bibitem{adhikari2018universality}
K.~Adhikari and K.~Saha.
\newblock Universality in the fluctuation of eigenvalues of random circulant matrices.
\newblock {\em Statistics \& Probability Letters}, 138:1--8, 2018.

\bibitem{BenaychGeorgesGuionnet2014}
F.~Benaych-Georges and A.~Guionnet.
\newblock Central limit theorem for eigenvectors of heavy tailed matrices.
\newblock {\em Electronic Journal of Probability}, 19:1--27, 2014.

\bibitem{benaych2014central}
F.~Benaych-Georges, A.~Guionnet, and C.~Male.
\newblock Central limit theorems for linear statistics of heavy tailed random matrices.
\newblock {\em Communications in Mathematical Physics}, 329(2):641--686, 2014.

\bibitem{benaych2016fluctuations}
F.~Benaych-Georges and A.~Maltsev.
\newblock Fluctuations of linear statistics of half-heavy-tailed random matrices.
\newblock {\em Stochastic Processes and their Applications}, 126(11):3331--3352, 2016.

\bibitem{jana2025cltles}
I.~Jana and S.~Rani.
\newblock Clt for les of correlated non-hermitian random matrices, 2025.
\newblock arXiv preprint arXiv:2503.22542.

\bibitem{jana2024spectrum}
I.~Jana and S.~Rani.
\newblock Spectrum of random centrosymmetric matrices; clt and circular law.
\newblock {\em Random Matrices: Theory and Applications}, 14(01):2450026, 2025.

\bibitem{male2017limiting}
C.~Male.
\newblock The limiting distributions of large heavy wigner and arbitrary random matrices.
\newblock {\em Journal of Functional Analysis}, 272(1):1--46, 2017.

\bibitem{o2016central}
S.~O’Rourke and D.~Renfrew.
\newblock Central limit theorem for linear eigenvalue statistics of elliptic random matrices.
\newblock {\em Journal of Theoretical Probability}, 29:1121--1191, 2016.

\bibitem{weaver1985centrosymmetric}
J.~R. Weaver.
\newblock Centrosymmetric (cross-symmetric) matrices, their basic properties, eigenvalues, and eigenvectors.
\newblock {\em The American Mathematical Monthly}, 92(10):711--717, 1985.

\bibitem{west2001introduction}
D.~B. West et~al.
\newblock {\em Introduction to graph theory}, volume~2.
\newblock Prentice hall Upper Saddle River, 2001.

\end{thebibliography}

\end{document}